\documentclass[11pt]{amsart}

\usepackage{amsmath,amssymb,amsfonts,amsthm,amstext,verbatim}
\usepackage{enumerate,color,graphicx,psfrag} 
\usepackage{url}
\usepackage{mdwlist}   
\usepackage{mathtools} 
\usepackage{wasysym}   

\usepackage{epstopdf} 
\usepackage{mathrsfs}

\allowdisplaybreaks

\newtheorem{thm}{Theorem}[section]
\newtheorem{lemma}[thm]{Lemma}
\newtheorem{prop}[thm]{Proposition}
\newtheorem{cor}[thm]{Corollary}

\newtheorem{definition}[thm]{Definition}

\newtheorem{assu}{Assumption}

\newtheorem{obs}[thm]{Observation}

\theoremstyle{remark} 
\newtheorem{rem}[thm]{Remark}

\theoremstyle{remark} 

\newtheorem{notrems}[thm]{Notational remarks}
\theoremstyle{remark} 
\newtheorem*{notrem*}{Notational remark}
\newtheorem*{notrems*}{Notational remarks}

\theoremstyle{plain}

\theoremstyle{plain}

\theoremstyle{plain}

\newcommand{\be}[1]{\begin{equation}\label{#1}}
\newcommand{\ee}{\end{equation}}

\newcommand{\ba}[1]{\begin{align}\label{#1}}
\newcommand{\ea}{\end{align}}

\newcommand{\ben}{\begin{equation*}}
\newcommand{\een}{\end{equation*}}

\newcommand{\ban}{\begin{align*}}
\newcommand{\ean}{\end{align*}}

\def\mik{1}
\newcommand\cpsfrag[2]{\ifnum\mik=1\psfrag{#1}{#2}\fi}

\newcommand{\calA}{\mathcal{A}}
\newcommand{\calB}{\mathcal{B}}
\newcommand{\calC}{\mathcal{C}}
\newcommand{\calD}{\mathcal{D}}
\newcommand{\calE}{\mathcal{E}}

\newcommand{\calM}{\mathcal{M}}
\newcommand{\calN}{\mathcal{N}}

\newcommand{\calQ}{\mathcal{Q}}

\newcommand{\calV}{\mathcal{V}}
\newcommand{\calW}{\mathcal{W}}

\newcommand{\frC}{\mathfrak{C}}

\newcommand{\frM}{\mathfrak{M}}

\newcommand{\bbC}{\mathbb{C}}

\newcommand{\bbL}{\mathbb{L}}

\newcommand{\bbP}{\mathbb{P}}

\newcommand{\bbT}{\mathbb{T}}

\newcommand{\bbZ}{\mathbb{Z}}

\newcommand\ve{\varepsilon}

\newcommand\om{\omega}       

\newcommand{\ra}{\rightarrow}
\newcommand{\supp}{\mathrm{supp}}
\newcommand{\xra[1]}{\xrightarrow{#1}}

\newcommand{\Ball}{\La}  
\newcommand{\Piv}{\mathrm{Piv}} 

\newcommand\TT{\bbT}           

\newcommand{\PP}{\mathbb{P}}     
\newcommand{\EE}{\mathbb{E}}     
\DeclareMathOperator{\Prob}{Prob\!} 

\newcommand{\ind}{\mathbf{1}} 

\newcommand\dist{\mathrm{dist}}    
\newcommand\diam{\mathrm{diam}}    
\newcommand\imag{\mathbf{i}}

\newcommand\resp{respectively}

\newcommand\ul{\underline}
\newcommand\ol{\overline}

\newcommand\xlra{\xleftrightarrow}

\newcommand{\RR}{\mathbb{R}}     
\newcommand{\CC}{\mathbb{C}}     
\newcommand{\NN}{\mathbb{N}}     
\newcommand{\ZZ}{\mathbb{Z}}     

\newcommand{\cl}{\mathcal{C}} 	 
\newcommand{\clColl}{\mathscr{C}} 
\newcommand{\muColl}{\mathscr{M}} 

\newcommand{\RS}{\hat{\CC}}

\newcommand{\hide}[1]{}

\title[]
{Conformal measure ensembles for percolation and the FK-Ising model}
\date{\today}

\author[F.~Camia]{Federico Camia}
\thanks{NYU Abu Dhabi and VU University, Amsterdam}

\author[R.P.~Conijn]{Ren\'{e} Conijn}
\thanks{Utrecht University}

\author[D.~Kiss]{Demeter Kiss}
\thanks{Statistical Laboratory, University of Cambridge and AIMR, Tohoku University, Sendai}

\begin{document}

\begin{abstract}
Under some general assumptions, we construct the scaling limit of open clusters and their 
associated counting measures in a class of two dimensional percolation models. Our results
apply, in particular, to critical Bernoulli site percolation on the triangular lattice and to the
critical FK-Ising model on the square lattice. 
Fundamental properties of the scaling limit, such as conformal covariance, are explored.
As an application to Bernoulli percolation, we obtain the scaling limit of the largest cluster
in a bounded domain. We also apply our results to the critical, two-dimensional Ising model,
obtaining the existence and uniqueness of the scaling limit of the magnetization field, as well
as a geometric representation for the continuum magnetization field which can be seen as a
continuum analog of the FK representation.
\end{abstract}

\maketitle


\tableofcontents

\section{Dedication and Synopsis}\label{sec:Synopsis}
This paper is dedicated to Chuck Newman on the occasion of his 70th birthday. The main goal of the paper is the study
of the continuum scaling limit of counting measures for percolation and FK-Ising clusters in two dimensions, leading to
the concept of \emph{conformal measure ensemble}. This is a fitting topic for a paper dedicated to Chuck since the concept
was originally discussed, in the context of the Ising model, by Chuck and the first author during a conversation in New York
about ten years ago. It was one of the early discussions in the project that led to \cite{CN09, C12, CGN15},
but somehow the term ``conformal measure ensemble'' never appeared
in print before this paper. Among other things, in this paper we complete the project started then and use the FK-Ising
conformal loop ensemble to express the continuum scaling limit of the critical Ising magnetization field as a sum of the
area measures of continuum FK-Ising clusters with fair, i.i.d.\ $\pm$ signs. This gives a geometric representation and
amounts to a continuum analog of the FK representation for the Ising magnetization. The key step consists in showing
that the collection of appropriately rescaled counting measures of Ising-FK clusters has a scaling limit;
with this result, we can give an alternative proof of the uniqueness and conformal covariance of the scaling limit of the critical
Ising magnetization, and also shows that the limiting magnetization is measurable with respect to the collection of loops
that describe the full scaling scaling limit of the FK-Ising process \cite{KS16}.

In the case of critical site percolation on the triangular lattice, we show that the collection of appropriately rescaled
counting measures of macroscopic clusters converges, in the scaling limit, to a collection of conformally covariant measures
whose supports are the scaling limits of the macroscopic clusters themselves. As a consequence, we obtain that the largest
percolation cluster in a bounded domain has a conformally covariant scaling limit. These results build on joint work of the
first author with Chuck \cite{Camia2006}.

Chuck's papers mentioned above are only a tiny and biased sample of his impressive production. His many contributions
to probability and statistical mechanics are both broad and profound. Through his work and teaching, Chuck has built an
important legacy, and one can only wish that he will continue to guide and inspire younger researchers for many years to come.

\subsection*{Acknowledgements}
The work of the first author was supported in part by the Netherlands Organization for Scientific Research (NWO) through grant Vidi 639.032.916.
The work of the second author is partly supported by NWO Top grant 613.001.403. When the research was carried out, the second author was at
VU University Amsterdam. The third author thanks NWO for its financial support and Centrum Wiskunde \& Informatica (CWI) for its hospitality
during the time when he was a PhD student, when the project was initiated. All three authors thank Rob van den Berg for fruitful discussions. The
first author thanks Chuck Newman for his friendship and invaluable guidance during many years, and for being a constant inspiration.

\section{Introduction}\label{sec:Introduction}
Several important models of statistical mechanics, such as percolation and the Ising and Potts models, can be
described in terms of clusters. In the last fifteen years, there has been tremendous progress in the study of the
geometric properties of such models in the scaling limit. Much of that work has focused on interfaces, that is, cluster
boundaries, taking advantage of the introduction of the Schramm-Loewner Evolution (SLE) by Oded Schramm in \cite{S00}.
In this paper, we are concerned with the scaling limit of the clusters themselves and their ``areas.'' More precisely, we
analyze the scaling limit of the collection of clusters and the associated counting measures (rescaled by an appropriate
power of the lattice spacing).

Our main results are valid under some general assumptions, which can be verified for Bernoulli site percolation on the
triangular lattice and for the FK-Ising model on the square lattice. Roughly speaking, our main results say that, under
suitable assumptions, in a general two-dimensional percolation model, the collection of clusters and their associated
counting measures, once appropriately rescaled, has a unique weak limit, in an appropriate topology, as the lattice
spacing tends to zero. The collection of clusters converges to a collection of closed sets (the ``continuum clusters''),
while the collection of rescaled counting measures converges to a collection of continuum measures whose supports
are the continuum clusters.

Our results are nontrivial at the critical point of the percolation model. For instance, in the case of critical site percolation on
the triangular lattice, where a scaling limit in terms of cluster boundaries is known to exist and to be conformally invariant
\cite{Camia2006} (it can be described in terms of SLE$_6$ curves), we show that the continuum clusters are also conformally
invariant, and that the associated measures are conformally covariant. The conformal covariance property of the collection of
measures is a consequence of the conformal invariance of the critical scaling limit. Because of this property, we use the expression
``conformal measure ensemble'' to denote the collection of measures arising in the scaling limit of a critical percolation model.

The scaling limit of the rescaled counting measures is in the spirit of \cite{Garban2010}, and indeed we rely heavily on techniques
and results from that paper. There is however a significant difference in that we distinguish between different clusters. In other words,
we don't obtain a single measure that gives the combined size of all clusters inside a domain, but rather a collection of measures, one
for each cluster. This is the main technical difficulty of the present paper. The reward is that handling individual clusters leads to new,
interesting applications to Bernoulli percolation and the Ising model, which represent the main motivation of the paper and are discussed
in detail in Section \ref{sec:app}.

In the case of Bernoulli percolation, our main application is the scaling limit of the largest clusters in a bounded domain.
The application to the Ising model requires a brief discussion. Consider a critical Ising model on the scaled lattice $\eta{\mathbb Z}^2$.
Using the FK representation, one can write the total magnetization in a domain $D$ as $\sum_i \sigma_i \nu^{\eta}_i(D)$,
where the $\sigma_i$'s are $(\pm 1)$-valued, symmetric random variables independent of each other and everything else, and
$\nu^{\eta}_i = \sum_{u \in {\mathcal C}_i}\delta_u$ is the counting measure associated to the $i$-th cluster ($\delta_u$
denotes the Dirac measure concentrated at $u$ and the order of the clusters is irrelevant) and
$\nu^{\eta}_i(D) = |{\mathcal C}_i \cap D|$, where ${\mathcal C}_i$ is the $i$-th cluster. The first author and Newman
\cite{CN09} noticed that the power of $\eta$ by which one should rescale the magnetization to obtain
a limit, as $\eta \to 0$, is the same as the power that should ensure the existence of a limit for the rescaled counting measures.
They then predicted that one should be able to give a meaning to the expression ``$\Phi^0=\sum_i \sigma_i \mu^{0}_i$,''
where $\Phi^0$ is the limiting magnetization field, obtained from the scaling limit of the renormalized lattice magnetization,
and $\{ \mu^{0}_i \}$ is the collection of measures obtained from the scaling limit of the collection of rescaled versions of the
counting measures $\{ \nu^{\eta}_i \}$. The existence and uniqueness of the limiting magnetization field was proved in \cite{CGN15};
here we complete the program put forward in \cite{CN09} for the two-dimensional critical Ising model by showing that the Ising
magnetization field can indeed be expressed in terms of cluster measures, thus providing a geometric representation (a sort of
continuum FK representation based on continuum clusters) for the limiting magnetization field. As a byproduct, we also obtain
a proof of the existence and uniqueness of the limiting magnetization field alternative to \cite{CGN15} and independent of \cite{CHI15}.

\subsection{Definitions and main results}
Let $\bbL$ denote a regular lattice with vertex set $V(\bbL)$ and edge set $E(\bbL)$. 
For $u$ and $v$ in $V(\bbL)$, we write $u \sim v$ if $(u,v) \in E(\bbL)$.
We are interested in Bernoulli percolation and FK-Ising percolation in $\mathbb{L}$ with parameter $p$. 
When we talk about FK-Ising percolation, $\bbL$ will be the square lattice ${\bbZ}^2$.
The FK clusters are defined as illustrated in Figure \ref{fig:FKClusters}, and we think of them as closed sets whose boundaries
are the loops in the medial lattice shown in Figure \ref{fig:FKClusters} (see \cite{G06} for an introduction to FK percolation).

\begin{figure}
	\centering 
	\includegraphics[width = 0.65\textwidth]{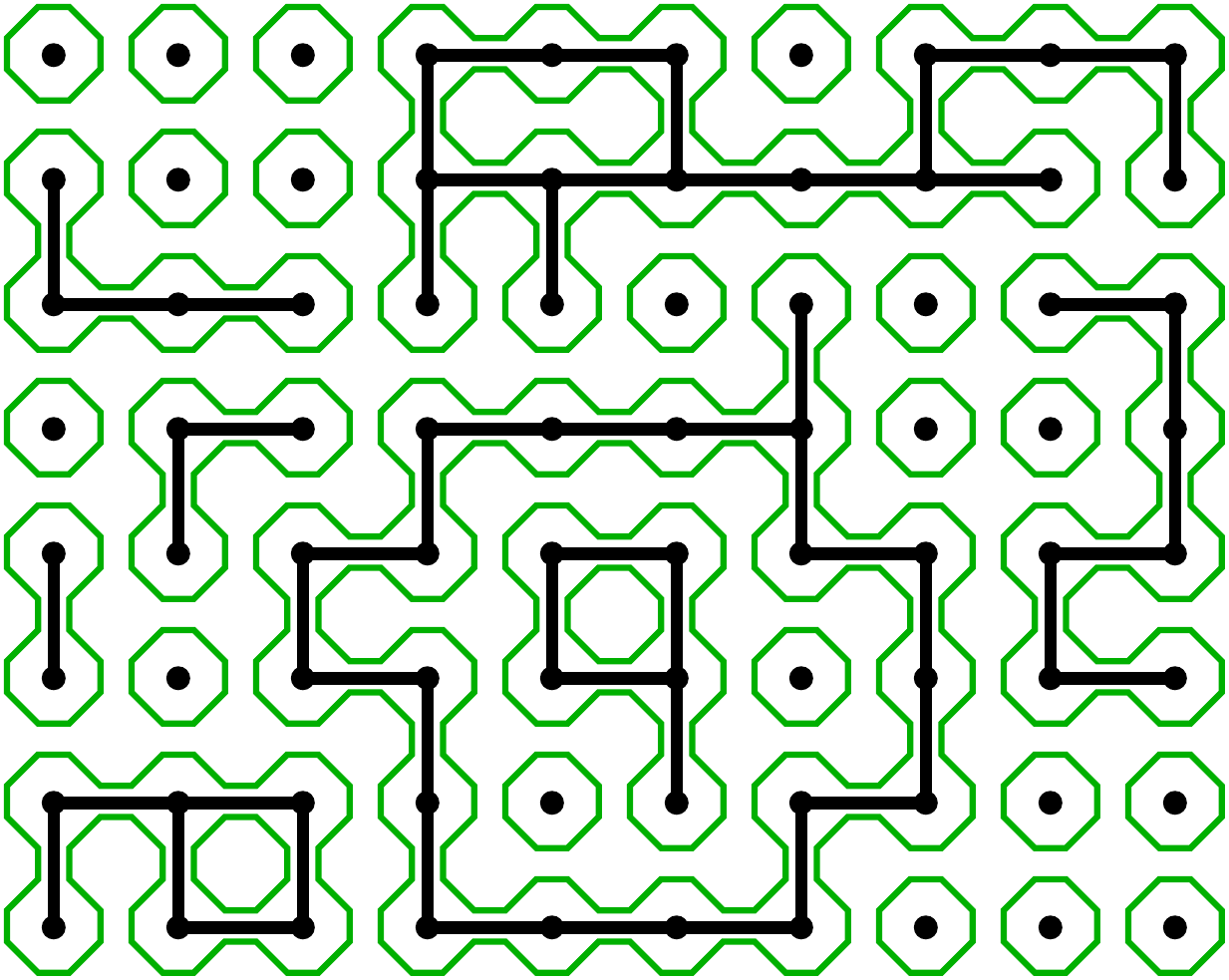}
	\caption{Illustration of FK clusters. Black dots represent vertices of ${\bbZ}^2$, black horizontal and vertical edges represent FK bonds.
The FK clusters are highlighted by lighter (green) loops on the medial lattice.}
	\label{fig:FKClusters}
\end{figure}

When dealing with Bernoulli percolation, $\bbL$ will be the triangular lattice $\TT$, with vertex set
\[ 
   V(\TT):=\left\{ x+y\epsilon\in\mathbb{C}\,\left|\, x,y\in\mathbb{Z}\right.\right\}, 
\]
where $\epsilon = e^{\pi \imag/3}$. The edge set $E(\TT)$ of $\TT$ consists of the pairs $u,v\in V$ 
for which $\left\Vert u-v\right\Vert _{2}=1$. Further, let $H_u$ denote the regular hexagon centered at
$u\in V(\TT)$ with side length $1/\sqrt{3}$ with two of its sides parallel to the 
imaginary axis. Clusters are connected components of open or closed hexagons
(see \cite{Grimmett1999} for an introduction to Bernoulli percolation).

Let $\eta>0$ and consider Bernoulli percolation on $\eta\mathbb{T}$ or the FK-Ising model on $\eta\mathbb{Z}^2$.
We think of open and closed clusters as compact sets. To distinguish between them, we will call open clusters `red'
and closed clusters `blue' (we deviate from the usual terminology of open and closed clusters on purpose:
we reserve the words `open' and `closed' to describe the topological properties of sets). Let $\sigma_\eta$
denote the union of the red clusters in $\eta\mathbb{L}$.

Further, let
\renewcommand{\Ball}{\Lambda}
\begin{align*}
  \Ball_r:=\{z\in \CC\,|\, |\Re z|\leq r,|\Im z|\leq r\}
\end{align*}
denote the ball of radius $r$ around the origin in the $L^\infty$ norm. We set $\Ball_r(u) = u + \Ball_r$.

Our aim is to understand the limit of the set $\sigma_\eta$ as $\eta$ tends to $0$. It is easy to see that
the limit of $\sigma_\eta$ in the Hausdorff topology as $\eta\rightarrow 0$ is trivial: it is the empty set when $p=0$
and a.s. $\CC$ for $p>0$. Hence we concentrate on the connected components, i.e. clusters, of $\sigma_\eta$ with diameter at least $\delta$ for some fixed $\delta >0$.
It is well-known (see for instance \cite{Aizenmann1987}) that, again, we get trivial limits unless $p = p_c$.
(For $p<p_c$ the limit of each of the clusters  is the empty set, while for $p>p_c$ the limit of the unique largest clusters is dense in $\CC$,
with the other clusters having the empty set as a limit.) Hence we consider $p = p_c$ in the following, and state informal versions of our main results
after some additional definitions. The precise versions of our results are postponed to later sections.

For a set $A\subset \CC$ and $u,v\in \CC$ we write $u\xleftrightarrow{A} v$ if there is a red path running in $A$ which connects
$u$ to $v$. When $A$ is omitted, it is assumed to be $\CC$. Let $\diam(A)$ denote the $L^\infty$ diameter of $A$.
For $u\in \eta V$ denote by $\cl^\eta(u)$ the connected component (i.e. cluster)
of $u$ in $\sigma_{\eta}$.
If $D$ is a simply connected domain with piecewise smooth boundary, we
let $\clColl_D^\eta(\delta)$ denote the collection of connected components of $\sigma_\eta$, 
which are contained in $D$ and have diameter larger than $\delta$. That is,
\begin{equation} \label{eq:def_clColl}
  \clColl_D^\eta(\delta) := \left\{\cl^\eta (u)\, |\, u\in \eta V,\,\cl^\eta (u) \subset D ,\, \diam(\cl^\eta(u)) \geq \delta \right\}.
\end{equation}
On many places $D$ is taken to be $\Ball_k$, in that case we simplify notation by
writing $\clColl_k^\eta(\delta) := \clColl_{\Ball_k}^\eta(\delta)$.
Finally let 
\begin{equation}\label{eq:def_full_clColl}
 \clColl^\eta(\delta) = \bigcup_{k\in \NN} \clColl_k^\eta(\delta)
\end{equation}
denote the collection of all connected components of $\sigma_\eta$
with diameter at least $\delta$.

In the following theorem, distances between subsets of $\CC$ will be measured by the Hausdorff distance built on the $L^\infty$ distance in $\CC$: For $A,B\subseteq \CC$,
\begin{align} \label{eq:def_Haus_top}
  d_H(A,B) := \inf \left\{ \varepsilon > 0 \,|\, A+\Ball_\varepsilon \supseteq B\text{ and } B+\Ball_\varepsilon \supseteq A\right\},
\end{align}
where $A+\Ball_{\varepsilon} := \{x+y\in \CC : x\in A,y\in \Ball_{\varepsilon} \}$. 

Let $\hat\CC$ be the one-point (Alexandroff) compactification of $\CC$, i.e. the Riemann sphere $\mathbb{\hat{C}}:=\mathbb{C}\cup\left\{ \infty\right\} .$
A distance between subsets of $\hat{\CC}$ which is equivalent to $d_H$ on bounded sets
is defined via the metric on $\CC$ with distance function
\begin{equation*}
 \Delta(u,v) := \inf_{\varphi} \int\frac{1}{1+|\varphi(s)|^2}ds,
\end{equation*}
where we take the infimum over all curves $\varphi(s)$ in $\CC$ from $u$ to $v$ and
$|\cdot|$ denotes the Euclidean norm.

The distance $D_H$ between sets is then defined by
\begin{equation}
 D_H(A,B) :=  \inf \left\{ \varepsilon > 0 \,|\, \forall u\in A: \exists v \in B: \Delta(u,v) \le \varepsilon \textrm{ and vice versa}\right\}.
\end{equation}

The distance between finite collections i.e., sets of subsets
of $\CC$, denoted by
$\mathscr{S},\mathscr{S}'$, is defined as
\begin{align} \label{eq:def_coll_dist}
  \min_{\phi} \max_{S\in\mathscr{S}} d_H(S,\phi(S))
\end{align}
where the infimum is taken over all bijections $\phi:\mathscr{S}\ra\mathscr{S}'$.
In case $|\mathscr{S}|\neq |\mathscr{S}'|$ we define the distance to be infinite.
To account for possibly infinite collections, $\mathscr{S}$ and $\mathscr{S}'$, of subsets of $\hat{\CC}$, we define
\begin{multline}
 dist(\mathscr{S},\mathscr{S}') \label{eq:def_coll_dist_hat_CC} \\
 \begin{aligned}
 \;\;\;\;\;\;\;\;\;\;&:= \inf \left\{ \varepsilon > 0 \,|\, \forall A\in \mathscr{S} \; \exists B \in \mathscr{S}': D_H(A,B) \le \varepsilon \textrm{ and vice versa}\right\}.
 \end{aligned}
\end{multline}
Convergence in the distance defined by \eqref{eq:def_coll_dist} implies
convergence in the distance $dist$, since the metrics $d_H$ and $D_H$ are
equivalent on bounded domains.

Our first result is the following, see Theorem \ref{thm:strong_Hausdorff} for a slightly stronger version.
\begin{thm}\label{thm:main_Hausdorff}
 Let $k > \delta>0$. Then, as $\eta \to 0$, $\clColl_k^\eta(\delta)$ converges in distribution, in the topology \eqref{eq:def_coll_dist},
 to a collection of closed sets which we denote by $\clColl_k^0(\delta)$. 
 Moreover, as $\delta \to 0$, $\clColl_k^0(\delta)$ has a limit in the metric \eqref{eq:def_coll_dist_hat_CC}, which we denote by $\clColl_k^0$.
\end{thm}

The next natural question is whether we can extract some more information from the scaling limit. In particular, can we
count the number of vertices in each of the clusters in $\clColl^\eta(\delta)$ in the limit as $\eta$ tends to $0$?
As we will see below, the number of vertices in the large clusters goes to infinity, hence we have to scale
this number to get a non-trivial result. The correct factor is $\eta^{-2}\pi_{1}^{\eta}(\eta,1),$ where $\pi_1^\eta(\eta,1)$
denotes the probability that $0$ is connected to $\partial \Ball_1$ in $\sigma_\eta$. We arrive at the informal formulation
of our next main result after some more notation.

For $S\subset \CC$ let $\mu^\eta_S$ denote the normalized counting measure of its vertices, that is,
\begin{equation} \label{eq:def_mu_C}
  \mu^\eta_S := \frac{\eta^2}{\pi_1^\eta(\eta,1)}\sum_{u\in S\cap \eta V}\delta_u,
\end{equation}
where $\delta_u$ denotes the Dirac measure concentrated at $u$. Further, let $\muColl_k^\eta(\delta)$ denote the collection of
normalized counting measures of the clusters in $\clColl_k^\eta(\delta)$. That is,
\begin{equation}
  \muColl_k^\eta(\delta) := \left\{\mu_\calC^\eta \,|\, \calC\in\clColl_k^\eta(\delta) \right\}.
\end{equation}
Similarly $ \muColl^\eta(\delta) := \{\mu_\calC^\eta \,|\, \calC\in\clColl^\eta(\delta) \}$.
We use the Prokhorov distance for the normalized counting measures. For finite Borel measures $\mu,\nu$ on $\CC$, it is defined as
\begin{multline}
  d_P(\mu,\nu)\nonumber\\
  \begin{aligned}
   \;\;\;&:= \inf\left\{\varepsilon > 0\,|\,\, \mu(S)\leq\nu(S^\varepsilon) + \varepsilon,\, \nu(S)\leq\mu(S^\varepsilon) +\varepsilon
  \text{ for all closed } S\subseteq \CC\right\}.
\end{aligned}
\end{multline}
where $S^\varepsilon = S+\Ball_\varepsilon$. Then we construct a metric on collections of Borel measures from $d_P$ similarly
to \eqref{eq:def_coll_dist}. We also introduce a distance $Dist$ between (infinite) collections of measures which is the same as
\eqref{eq:def_coll_dist_hat_CC} but with collections of sets replaced by collections of measures and with the distance $D_H$
replaced by the Prokhorov distance $d_P$.

We arrive at the following result (see Theorem \ref{thm:strong_meas} for a slightly stronger version).
\begin{thm}\label{thm:main_meas}
  Let $k > \delta>0$, then $\muColl_k^\eta(\delta)$ converges in distribution to a
  collection of finite measures which we denote by $\muColl_k^0(\delta)$. 
  Moreover, as $\delta \to 0$, $\muColl_k^0(\delta)$ has a limit in the metric $Dist$, which we denote by $\muColl_k^0$.
\end{thm}

The next theorem is a full-plane analogue of Theorems \ref{thm:main_Hausdorff} and \ref{thm:main_meas}.

\begin{thm}\label{thm:main_Hausdorff_fullPlane}
Let $\bbP_k$ denote the joint distribution of $(\clColl_k^0,\muColl_k^0)$. 
There exists a probability measure $\bbP$ on the space of collections of subsets of $\hat{\CC}$
and collections of measures, which is the full plane limit of the probability measures $\bbP_k$ in the sense that,
for every bounded domain $D$, the restriction $\bbP_k|_D$ of $\bbP_k$ to $(\clColl_D^0,\muColl_D^0)$ converges to the
restriction $\bbP|_D$ of $\bbP$ to $(\clColl_D^0,\muColl_D^0)$ as $k \to \infty$.
\end{thm}

The next theorem shows that the collections of clusters and measures from the previous theorem are invariant under
rotations and translations, and transform covariantly under scale transformations. (The theorem could be extended to
include more general fractal linear (M\"obius) transformations by restricting to the Riemann sphere minus a neighborhood
of infinity and its preimage. For simplicity, we restrict attention to linear transformations that map infinity to itself.)
The random variables with distribution $\PP$ introduced in the previous theorem are
denoted by $(\clColl^0,\muColl^0)$.
\begin{thm}\label{thm:invar}
  Let $f$ be a linear map from $\bbC$ to $\bbC$, that is $f(z) = rz+t$ with $r,t \in \CC$.
  Assume that
\begin{align*}
\lim_{\eta \to 0} \pi^{\eta}_1(a,b) = \left(\frac{a}{b}\right)^{\alpha_1+o(1)}
\end{align*}
for all $b>a>\eta$ and some $\alpha_1 \in [0,1]$, where $o(1)$ is understood as $b/a \to \infty$. We set
  \begin{align*}
    f(\clColl^0) & := \{f(\cl)\,:\, \cl\in\clColl^0\}, \text{ and}\\
    f(\muColl^0) & := \{\mu^{0*}\,:\,\mu^0\in\muColl^0\}
  \end{align*}
  where $\mu^{0*}$ is the modification of push-forward measure of $\mu^0$ along $f$ defined as 
  \begin{align*}
    \mu^{0*}(B) := |r|^{2-\alpha_1}\mu^{0}(f^{-1}(B))
  \end{align*}
  for Borel sets $B$.  
  Then the pairs $(f(\clColl^0), f(\muColl^0))$ and $(\clColl^0,\muColl^0)$ have the same distribution.
\end{thm}

 \begin{rem}
 In the case of Bernoulli percolation,
 we will prove invariance/covariance under all conformal maps between any two bounded domains with
 piecewise smooth boundaries (see Theorems \ref{thm:cinvar}, \ref{thm:cinvar_prec} and \ref{thm:covar_prec}).
 \end{rem}

\subsection*{Organization of the paper}
In the next section we discuss some applications of our results. First we consider
applications to Bernoulli percolation on the triangular lattice.
Secondly we provide a geometric representation for the magnetization field of the critical Ising model
in terms of FK clusters.

In Section \ref{sec:preliminaries} we introduce the main tools and assumptions which we use throughout the paper,
namely the loop process, the quad-crossing topology,
arm events and the general assumptions
under which we prove our main results. We finish Section \ref{sec:preliminaries} with checking that
the assumptions hold for critical Bernoulli percolation on $\TT$
and for the critical FK-Ising model on the square lattice. 
In Sections \ref{sec:approx}-\ref{sec:cont_meas} we give precise versions
and proofs of Theorems \ref{thm:main_Hausdorff}, \ref{thm:main_meas} and
\ref{thm:main_Hausdorff_fullPlane}.

We investigate some fundamental properties of the continuum clusters
and their normalized counting measures in Section \ref{sec:props}. In particular,
we also discuss the conformal
invariance and covariance properties of the clusters in this section.
We finish the paper with Section \ref{sec:pf_of_conv_largest_clusters} where we prove the convergence of
the largest clusters for Bernoulli percolation in a bounded domain.

\section{Applications} \label{sec:app}

\subsection{Largest Bernoulli percolation clusters and conformal invariance/covariance} \label{ssec:pf_cor_largest}

Our first application concerns the scaling limit of the largest percolation clusters in a bounded domain with closed (blue) boundary condition.
Denote by $\calM_{(i)}^{\eta}$ the $i$-th largest cluster in $\Ball_{1}\cap \sigma_{\eta}$, where we measure clusters
according to the number of vertices they
contain. 

In a sequence of papers, the behavior of the normalized number of vertices,
\begin{align} \label{eq:largest_volume}
	\frac{|\calM_{(i)}^\eta|}{\eta^{-2}\pi_1^\eta(\eta,1)} = \mu_{\calM_{(i)}^\eta}^\eta\!\!\!(\Ball_1) ,
\end{align}
was investigated for $\eta>0$ and $i\geq1$. Probably the first such results appeared in \cite{BCKS99} and \cite{BCKS01}.
Using Theorems \ref{thm:main_Hausdorff} and \ref{thm:main_meas}
and results in Section \ref{sec:clustersinSubset} about convergence of clusters and portions of clusters in bounded domains,
we deduce the following theorem.

\begin{thm}\label{thm:largest_Clusters}
  For all $i\in \NN$, the cluster $\calM_{(i)}^\eta$ and its normalized counting
  measure $\mu_{\calM_{(i)}^{\eta}}^{\eta}$ converge in distribution to a closed
  set $\calM^0_{(i)}$ and a measure $\mu_{\calM_{(i)}^0}^0$, respectively, as $\eta\ra0$.
\end{thm}

\begin{proof}
 The result follows directly from Theorem \ref{thm:conv_Largest_clusters_precise} in Section \ref{sec:pf_of_conv_largest_clusters}.
\end{proof}

Recently some of the results from \cite{BCKS99,BCKS01} where sharpened \cite{BC12,BC13,K13}.
These sharpened results, in combination with Theorem \ref{thm:largest_Clusters}, imply that the distribution
of $\mu_{\calM_{(i)}^0}^0\!\!\!(\Ball_1)$ has no atoms \cite{BC13}, that its support is $(0,\infty)$ \cite{BC12} and
that it has a stretched exponential upper tail \cite{K13}.  

It is a celebrated result of Smirnov \cite{Smirnov2001a} that critical site percolation on the triangular lattice is conformally
invariant in the limit as $\eta\ra0$. See also \cite{CaNe07,Camia2006}. As we will show, under certain technical conditions,
this implies that the collections of large clusters in the limit as $\eta\ra0$ are also conformally invariant, while their
normalized counting measures are conformally covariant by the results in \cite{Garban2010}.
We denote by $\mathscr{B}_{D}^{\eta}(\delta)$ the collection of clusters, with diameter
greater than $\delta>0$, in a domain $D$ with closed boundary condition.
In Section \ref{sec:clustersinSubset} we will see that, as $\eta \to 0$, this collection
converges in distribution to a limiting collection of clusters $\mathscr{B}_D^0(\delta)$.
The latter converges as $\delta$ tends to $0$ to the random collection $\mathscr{B}_D^0$. To indicate that we 
consider the measures of the clusters in $\mathscr{B}_D^0$ instead of the clusters in $\clColl_D^0$ we add
a tilde, for example the collection of measures of the clusters in $\mathscr{B}_{\Ball_1}^0$ is denoted by
$\tilde{\muColl}_{\Ball_1}^0$. We obtain the following result, which is stated in a slightly stronger form as
Theorems \ref{thm:cinvar_prec} and \ref{thm:covar_prec}. 

\begin{thm}\label{thm:cinvar}
  Let $f$ be a conformal map defined on an open neighbourhood of $\Ball_1$, and $D = f(\Ball_1)$. We set
  \begin{align*}
    f(\mathscr{B}_{\Ball_1}^0) & := \{f(\calB)\,:\, \calB\in\mathscr{B}_{\Ball_1}^0\}, \text{ and}\\
    f(\tilde{\muColl}_{\Ball_1}^0) & := \{\mu^{0*}\,:\,\mu^0\in\tilde{\muColl}_{\Ball_1}^0\}
  \end{align*}
  where $\mu^{0*}$ is the modification of the push-forward measure of $\mu^0$ along $f$ defined as 
  \begin{align*}
    \mu^{0*}(B) := \int_{f^{-1}(B)} |f'(z)|^{91/48}d\mu^{0}(z)
  \end{align*}
  for Borel sets $B$.  
  Then the pairs $(f(\mathscr{B}_{\Ball_1}^0), f(\tilde{\muColl}_{\Ball_1}^0))$ and
  $(\mathscr{B}^0_D,\tilde{\muColl}^0_D)$ have the same distribution.
\end{thm}

\begin{proof}
 The result follows from a combination of Theorems \ref{thm:cinvar_prec} and \ref{thm:covar_prec}, stated
and proved in Section \ref{ssec:cinvar}.
\end{proof}

\subsection{Geometric representation of the critical Ising magnetization field} \label{ssec:crit_Ising_constr}

In this section we give an alternative proof of the existence and uniqueness of the limiting magnetization field obtained by
taking the critical scaling limit of the magnetization in the two-dimensional Ising model, a result first proved in \cite{CGN15}.
We also prove a geometric representation for the scaling limit of the critical Ising magnetization in two dimensions that was
first conjecture in \cite{CN09}. There, it was heuristically argued that the Ising magnetization field should be expressible in
terms of the limiting cluster measures of the FK-Ising clusters, giving a sort of continuum FK representation based on continuum
clusters;
here, we provide the proof of a precise statement to that effect (see Theorem~\ref{thm:Ising} below).

Consider a two-dimensional critical Ising model on $\eta {\mathbb Z}^2$ and its FK representation (see, e.g., \cite{G06}).

In what follows, we will assume Wu's celebrated result on the power law decay of the critical Ising two-point function
\cite{W66}. This assumption implies that, for critical FK-Ising percolation, $\eta^2/\pi^{\eta}(\eta,1)$ behaves like $\eta^{15/8}$
as $\eta \to 0$. (See Remark~1.5 of \cite{CGN15} for a discussion of this point.)
We denote by $\Phi^{\eta}$ the lattice magnetization field, defined as
\begin{equation*}
\Phi^{\eta}:= \eta^{15/8} \, \sum_{x \in \eta\, {\mathbb Z}^2} S_x \delta_x \, ,
\end{equation*}
where $S_x$ is the spin at $x$ and $\delta_x$ is the Dirac delta at $x$.
We also introduce the $\varepsilon$-cutoff magnetization $\Phi^{\eta}_{\varepsilon}$, define as
\begin{align*}
\Phi^{\eta}_{\varepsilon} := \sum_{\cl \in \clColl^\eta(\varepsilon)} {\mathcal S}_\cl \mu^{\eta}_{\cl} \, ,
\end{align*}
where the sum is over all FK-Ising clusters of diameter larger than $\varepsilon$ (the order of the sum is irrelevant), the ${\mathcal S}_\cl$'s
are i.i.d. symmetric $(\pm1)$-valued random variables, the $\mu^{\eta}_{\cl}$'s are the FK-Ising normalized counting
measures, and we think of $\Phi^{\eta}_{\varepsilon}$ as a random signed measure acting on the space $C^{\infty}_{0}$ of infinitely
differentiable functions with bounded support.
Note that, if ${\bf 1}_L$ denotes the indicator function of $[-L,L]^2$, $\left\langle \Phi^{\eta}_{\varepsilon}, {\bf 1}_L \right\rangle$
is the magnetization in $[-L,L]^2$ produced by FK clusters of diameter at least $\varepsilon$.

\begin{lemma}\label{lemma:Ising}
For each $f \in C^{\infty}_{0}$, as $\eta \to 0$, $\left\langle \Phi^{\eta}_{\varepsilon}, f \right\rangle$ converges in distribution
to the random variable
\begin{align*}
\Phi^{0}_{\varepsilon}(f) := \sum_{\cl \in \clColl^0(\varepsilon)} {\mathcal S}_\cl \langle \mu^{0}_{\cl},f \rangle \, .
\end{align*}
\end{lemma}

\begin{proof}
The statement follows from Theorem \ref{thm:main_Hausdorff_fullPlane} by taking any $k$ such that the domain of $f$ is contained in $\Lambda_k$.
\end{proof}

\begin{thm}\label{thm:Ising}
If $f$ is a bounded function of bounded support, as $\eta \to 0$, then $\left\langle \Phi^{\eta}, f \right\rangle$ converges in
distribution to a  random variable $\Phi^{0}(f)$ measurable with respect to the collection of loops and signs, and such that
\begin{align*}
||\Phi^{0}(f) - \Phi^{0}_{\varepsilon}(f)||_2 \leq C \, ||f||_{\infty}^2 \, \varepsilon^{7/4}
\end{align*}
for any $\varepsilon>0$ and some positive constant $C<\infty$ independent of $f$.
Moreover, if $(f_n)_{n \in {\mathbb N}}$ is a sequence of bounded functions of bounded support converging to $f$ in the
sup-norm as $n \to \infty$, then $\Phi(f_n) \to \Phi(f)$ in $L_2$ as $n \to \infty$.
\end{thm}

\begin{proof}
We first identify a candidate for the limit $\Phi^{0}(f)$ of $\left\langle \Phi^{\eta}, f \right\rangle$.
To that end, we consider $\Phi^0_{\varepsilon}(f)$ as an element of $L_2(\Omega, {\mathbb P})$ and let $\varepsilon \to 0$.
The existence of the limit can be checked easily by considering sequences $\varepsilon_n \searrow 0$ and showing that
$(\Phi^0_{\varepsilon_n}(f))_n$ is a Cauchy sequence. For any $m>n$, 
denoting by $||\cdot||_2$ the $L_2$-norm and using ${\mathbb E}$ for expectation with respect to $\mathbb P$,
using the argument in the proof of Proposition 6.2 of \cite{C12}, we have that
\begin{eqnarray*}
|| \Phi^0_{\varepsilon_n}(f) - \Phi^0_{\varepsilon_m}(f) ||^2_2 & = &
{\mathbb E} \left( | \sum_{\cl \in \clColl^0(\varepsilon_m)\setminus \clColl^0(\varepsilon_n)} {\mathcal S}_\cl\, \mu^{0}_\cl(f) |^2 \right) \\
& \leq & \limsup_{\eta \to 0}{\mathbb E} \left( \sum_{\cl \in \clColl^\eta(\varepsilon_m)\setminus \clColl^\eta(\varepsilon_n)}
(\mu^{0}_\cl(f))^2 \right) \\ 
& \leq & \limsup_{\eta \to 0}
\mathbb{E}\left[ \sum_{\cl \in \clColl^\eta \setminus \clColl^\eta(\varepsilon_n)} (\mu^\eta_\cl(f))^2 \right] \\
& < & C \left(\sup_{x \in D}|f(x)|\right)^2 \varepsilon_n^{7/4} \, ,
\end{eqnarray*}
for some positive constant $C<\infty$ independent of $f$. Therefore, if $f$ is a bounded function of bounded support,
$\Phi^0_{\varepsilon}(f)$ converges, as $\varepsilon \to 0$ to an element $\Phi^0(f)$ of $L_2(\Omega,{\mathbb P})$;
moreover, for any $\varepsilon>0$,
\begin{align} \label{epsilon-upper-bound}
||\Phi^{0}(f) - \Phi^{0}_{\varepsilon}(f)||_2 \leq C \, ||f||_{\infty}^2 \, \varepsilon^{7/4} \, .
\end{align}

Using the triangle inequality, for any $\eta>0$, we can write
\begin{eqnarray*}
|| \Phi^0(f) - \left\langle \Phi^{\eta},f \right\rangle ||_{2} & \leq &
||\Phi^{0}(f) - \Phi^{0}_{\varepsilon}(f)||_2 \\
& + & ||\Phi^{0}_{\varepsilon}(f) - \Phi^{\eta}_{\varepsilon}(f)||_2 \\
& + & ||\Phi^{\eta}_{\varepsilon}(f) - \left\langle \Phi^{\eta},f \right\rangle||_2 \, .
\end{eqnarray*}
As $\eta \to 0$, the first term in the right hand side of the last inequality tends to zero because of \eqref{epsilon-upper-bound}.
The third term can be made arbitrarily small by letting $\eta \to 0$ and taking $\varepsilon$ small. Like \eqref{epsilon-upper-bound},
this follows from results and calculations in \cite{CN09} and from the proof of Proposition 6.2 of \cite{C12}.
For fixed $\varepsilon>0$, the remaining term can be expressed as a finite sum, containing the normalized counting measures of
clusters of diameter larger than $\varepsilon$ that intersect the support of $f$. As $\eta \to 0$, this term tends to zero because of
the convergence in probability of normalized counting measures proved in Theorem 7.2 under Assumption IV, and the $L_3$ bounds
provided by Lemma 3.15.

The $L_2$-continuity of $\Phi^0(\cdot)$ follows from
\begin{multline}
|| \Phi^0(f) - \Phi^0(f_n) ||^2_2\nonumber\\
\begin{aligned}
 & \leq & || \Phi^0(f) - \Phi^0_{\varepsilon}(f) ||^2_2 + || \Phi^0_{\varepsilon}(f) - \Phi^0_{\varepsilon}(f_n) ||^2_2
+  || \Phi^0_{\varepsilon}(f_n) - \Phi^0(f_n) ||^2_2 \nonumber\\
& \leq & || f-f_n ||^2_{\infty} \, {\mathbb E}\left(
\sum_{\cl \in \clColl^0(\varepsilon)} (\mu^{0}_\cl({\bf 1}_L))^2 \right) + C (|| f ||^2_{\infty} + || f_n ||^2_{\infty}) \varepsilon^{7/4}\nonumber \, ,
  \end{aligned}
\end{multline}
where ${\bf 1}_L$ denotes the indicator function of $[-L,L]^2$ and $L$ is such that supp$(f),\, $supp$(f_n) \subset [-L,L]^2$.
The fact that the term ${\mathbb E}\left(\sum_{\cl \in \clColl^0(\varepsilon)} (\mu^{0}_\cl({\bf 1}_L))^2 \right)$ is bounded follows,
for instance, from Proposition B.2 of \cite{CGN15}.

To conclude the proof of the theorem we note that, for every $\varepsilon>0$, the sum defining $\Phi^0_{\varepsilon}(f)$ is a.s.\
finite; therefore $\Phi^0_{\varepsilon}(f)$ is measurable with respect to the collections of loops and signs. Since $\Phi^0(f)$ is the
limit of $\Phi^0_{\varepsilon}(f)$ as $\varepsilon \to 0$, it is also measurable with respect to the collections of loops and signs.
\end{proof}

In the corollary below we consider the probability space $\Omega$ of continuum Ising-FK loops (CLE$_{16/3}$), clusters and area measures,
together with the random signs assigned to the clusters. An element of that space is denoted $\omega$ and the joint probability distribution
is denoted by $\mathbb P$.
We let $D$ be the space of infinitely differentiable functions with compact support equipped with the
topology of uniform convergence of all derivatives, whose topological dual $D'$ consists of all generalized functions. We remind the reader that,
according to Theorem \ref{thm:Ising}, the magnetization field $\Phi^0$ is measurable with respect to $\omega$.
\begin{cor}\label{cor:Ising}
There exists a random, continuous, linear functional $T \in D'$ with characteristic function
$\chi_T(f) = \int_{\Omega} e^{i\Phi^0(f)} d{\mathbb P}(\omega)$, for all $f \in D$.
\end{cor}

\begin{proof}
Since $D$ is a nuclear space, we can apply the Bochner-Minlos theorem (see, for example,
Theorem~3.4.2 on p.~52 of \cite{GF81}---a proof can be found in~\cite{GV64}). We define
\begin{equation} \label{eq-charac-func}
\chi(f) := \int_{\Omega} e^{i\Phi^0(f)} d{\mathbb P}(\omega)
\end{equation}
and check the following properties of $\chi$.
\begin{enumerate}
\item Normalization: $\chi(0)=1$.
\item Positivity: $\sum_{k,l=1}^m c_k \overline{c_l} \chi(f_k-f_l) \geq 0$
for every $m \in {\mathbb N}$, $f_1,\ldots,f_m \in D$ and
$c_1,\ldots,c_m \in {\mathbb C}$.
\item Continuity: $\chi(f) \to 0$ as $f \to 0$ (in the topology of $D$).
\end{enumerate}
The first property is clear. To establish the second property, let
$F = \sum_{k=1}^m c_k e^{i\Phi^0(f_k)}$ and note that the square of
the $L_2(\Omega,{\mathbb P})$-norm of $F$ is
\begin{eqnarray}
0 \leq ||F||_2^2 & = & \int_{\Omega} \sum_{k,l=1}^m c_k
\overline{c_l}
e^{i \Phi^0(f_k-f_l)} d{\mathbb P}(\omega) \\
& = & \sum_{k,l=1}^m c_k \overline{c_l} \chi(f_k-f_l).
\end{eqnarray}

The remaining step is to establish continuity of $\chi$. We think of
${\Phi^0(f)}$ as a sequence of random variables indexed by $f \in D$
as $f \to 0$ in the topology of $D$, which implies uniform
convergence of $f$ to zero. We have
\begin{equation} \label{eq-bound}
|| \Phi^0(f) ||^2_2 \leq || f ||^2_{\infty} {\mathbb E} \left(
\sum_\cl (\mu^{0}_\cl({\bf 1}_L))^2 \right),
\end{equation}
where ${\bf 1}_L$ denotes the indicator function of $[-L,L]^2$ and
$L$ is such that supp$(f) \subset [-L,L]^2$. This implies
convergence of $\Phi^0(f)$ to 0 in $L_2$ as $f \to 0$, which implies
convergence in probability, which implies convergence in
distribution, which is equivalent to pointwise convergence of
characteristic functions, which gives us the type of continuity we
need. Therefore, by an application of the Bochner-Minlos theorem,
there exists a random, continuous, linear functional $T \in D'$ with
characteristic function $\chi_{T}(f)=\chi(f)$.
\end{proof}

A result related to our Theorem \ref{thm:Ising} was recently proved by Miller, Sheffield and Werner \cite{MSW16}. They showed
(see Theorem 7.5 of \cite{MSW16}) that forming clusters of CLE$_{16/3}$ loops by a percolation process with parameter $p=1/2$
generates CLE$_{3}$, the Conformal Loop Ensemble with parameter $3$. CLE$_3$ describes the full scaling limit of Ising spin-cluster
boundaries \cite{BH16} while CLE$_{16/3}$, as already mentioned, describes the full scaling limit of Ising-FK cluster boundaries
\cite{KS16}. We note that, although the magnetization can obviously be expressed using Ising spin clusters, as a sum of their signed
areas, such a representation does not appear to be useful in the scaling limit because the area measures of spin clusters don't scale
like the magnetization. The usefulness of the representation in terms of FK clusters is due to the fact that both the FK clusters and the
magnetization need to be normalized by the same scale factor in the scaling limit. That is not true of the magnetization and the spin clusters.

\section{Further notation and preliminaries} \label{sec:preliminaries}

Above we interpreted the union of red hexagons in a percolation configuration $\sigma_\eta$, as a (random) subset of $\CC$. In
what follows, as an intermediate step, we will consider a percolation configuration as a (random) collection of loops.
These loops form the boundaries of the clusters. We will describe this space in Subsection \ref{ssec:def_loops}.
In order to define the clusters as subsets of the plane, we will also consider the (random) collection of
quads (`topological squares' with two marked opposing sides) which are crossed horizontally. This leads us to
the Schramm-Smirnov \cite{Smirnov2011} topological space, which we briefly recall in the second subsection.

\subsection{Space of nonsimple loops} \label{ssec:def_loops}
The random collection of loops will be denoted by $L_\eta$ for $\eta \ge 0$.
The distance between two curves $l, l'$ is defined as
\begin{equation} \label{eq:def_loopdist}
 d_c(l,l') := \inf \sup_{t \in [0,1]} \Delta(l(t), l'(t)),
\end{equation}
where the infimum is over all parametrizations of the curves.
The distance between closed sets of curves is defined similarly to the distance
$dist$ defined in \eqref{eq:def_coll_dist_hat_CC}
between collections of subsets of the Riemann sphere $\hat{\CC}$.
The space of closed sets of loops is a complete separable metric space.

For $\eta > 0$ the collection of (oriented) boundaries of the red clusters in $\sigma_\eta$
is the closed set of loops, denoted by $L_\eta$.
This set converges in distribution to $L_0$, called the \emph{continuum nonsimple loop process} \cite{Camia2006}.

\subsection{Space of quad-crossings}
We borrow the notation and definitions from \cite{Garban2010}. Let
$D\subset\RS$
be open. A quad $Q$ in $D$ is a homeomorphism $Q:\left[0,1\right]^{2} \ra Q([0,1]^2)\subseteq D$. Let $\calQ_{D}$ be the set of all quads, which we
equip with the supremum metric
\[
  d\left(Q_{1},Q_{2}\right)=\sup_{z\in\left[0,1\right]^{2}}\left|Q_{1}\left(z\right)-Q_{2}\left(z\right)\right|
\]
for $Q_{1},Q_{2}\in\calQ_{D}.$ 

A crossing of a quad $Q$ is a closed connected subset of $Q\left(\left[0,1\right]^{2}\right)$
which intersects $Q\left(\left\{ 0\right\} \times\left[0,1\right]\right)$
as well as $Q\left(\left\{ 1\right\} \times\left[0,1\right]\right).$
The crossings induce a natural partial order denoted by $\leq$ on
$\calQ_{D}.$ We write $Q_{1}\leq Q_{2}$ if all the crossings
of $Q_{2}$ contain a crossing of $Q_{1}.$ For technical reasons,
we also introduce a slightly less natural partial order on $\calQ_{D}:$
we write $Q_{1}<Q_{2}$ if there are open neighbourhoods $\calN_{i}$
of $Q_{i}$ such that for all $N_{i}\in\calN_{i},$ $i\in\left\{ 1,2\right\} ,$
$N_{1}\leq N_{2}.$ We consider the collection of all closed hereditary subsets
of $\calQ_{D}$ with respect to $<$ and denote it by$\mathscr{H}_{D}.$
It is the collection of the closed sets $\mathcal{S}\subset\calQ_{D}$
such that if $Q\in\mathcal{S}$ and $Q'\in\calQ_{D}$ with $Q'<Q$
then $Q'\in\mathcal{S}.$ 

For a quad $Q\in\calQ_{D}$ let $\boxminus_{Q}$ denote the
set
\begin{align*}
  \boxminus_{Q}: & =\left\{ \mathcal{S}\in\mathscr{H}_{D}\,\left|\, Q\in\mathcal{S}\right.\right\},
\end{align*}
which corresponds with the configurations where $Q$ is crossed.
For an open subset $\mathcal{U}\subset\calQ_{D}$ let $\boxdot_{\mathcal{U}}$
denote the set
\[
  \boxdot_{\mathcal{U}}:=\left\{ \mathcal{S\in\mathscr{H}_{D}\,\left|\,\mathcal{U}\cap\mathcal{S}=\emptyset\right.}\right\} ,
\]
which corresponds with the configurations where none of the quads of $\mathcal{U}$ is crossed.
We endow $\mathscr{H}_{D}$ with the topology $\mathscr{T}_{D}$ which
is the minimal topology containing the sets $\boxminus_{Q}^{c}$
and $\boxdot_{\mathcal{U}}^{c}$ as open sets for all $Q\in\calQ_{D}$
and $\mathcal{U}\subset\calQ_{D}$ open. We have:
\begin{thm} [Theorem 1.13 of \cite{Smirnov2011}]
  Let $D$ be an open subset
  of $\RS.$ Then the topological space $\left(\mathscr{H}_{D},\mathscr{T}_{D}\right)$
  is a compact metrizable Hausdorff space.
\end{thm}
Using this topological structure, we construct the Borel $\sigma$-algebra
on $\mathscr{H}_{D}.$ We get:

\begin{cor} [Corollary 1.15 of \cite{Smirnov2011}] \label{cor:comp_prob_meas_space}
  $\Prob\left(\mathscr{H}_{D}\right),$ the space of Borel probability measures of $\left(\mathscr{H}_{D},\mathscr{T}_{D}\right)$,
  equipped with the weak{*} topology is a compact metrizable Hausdorff space.
\end{cor}

\begin{notrems}
  \begin{enumerate}[i)]
    \item In the following we abuse the notation of a quad $Q$. When we refer to $Q$ as a subset of $\RS$, we consider its range $Q([0,1]^2)\subset\hat\CC$.
    \item \label{notrem:def_omega} Note that a percolation configuration $\sigma_\eta$, as defined in the
    introduction, naturally induces a quad-crossing configuration $\omega_{\eta}\in \mathscr{H}_{\hat\CC}$, namely
      \begin{align} \label{eq:def_omega}
        \omega_{\eta}:=\left\{Q\in \calQ_{\hat \CC}\,|\, \sigma_\eta \text{ contains a crossing of } Q \right\}.
      \end{align}
      Furthermore, $\PP_{\eta}$ will denote the law governing $(\omega_{\eta}\times L_\eta)$.
  \end{enumerate}
\end{notrems}

Further we will need the following definitions for restrictions of the configuration to a subset of the Riemann Sphere.
\begin{definition} \label{def:restr}
   Let $D\subseteq\hat{\CC}$ be an open set and $\om\in \mathscr{H}_{\hat{\CC}}$. Then $\om|_D$, the restriction of $\om$ to $D$,
   is defined as
   \begin{align*}
     \om|_D:=\{Q\in\om\,:\,Q\subset D\}.
   \end{align*}
   The image of $\om|_D$ under a conformal map $f:D\ra\hat{\CC}$ is defined as
   \begin{align*}
     f(\om|_D):=\{f(Q):\, Q\in\om|_D\}\in \mathscr{H}_{f(D)}.
   \end{align*}
   The restriction of the loop process to $D$ is defined as
   \begin{align*}
    L|_D := \{l\,:\,\exists \tilde{l} \in L \textrm{ s.t. } l \textrm{ is an excursion of } \tilde{l} \textrm{ in } D\}.
   \end{align*}
   The image of $L|_D$ under a conformal map $f:D\ra\hat{\CC}$ is defined as
   \begin{align*}
     f(L|_D):=\{f(l):\, l\in L|_D\}.
   \end{align*}
   Furthermore,  $\PP_{\eta,D}$ denotes the law of $(\om_{\eta,D}, L_{\eta,D}) := (\om_\eta|_D, L_\eta|_D)$ for $\eta\geq0$.

\end{definition}

\subsection{Assumptions} \label{ssec:assumptions}
Below we list the assumptions which are used throughout the article.

The edge set in the sublattice on $D \subset \CC$ of $\eta\bbL$ is
$(\eta E(\bbL))|_D := \{(u,v) \in \eta E(\bbL)\,:\,u,v \in \eta V(\bbL) \cap D\}$.
The discrete boundary of $D \subset \CC$ of the lattice $\eta\bbL$ is defined by:
\[
 \partial_{\eta} D := \{ u \in \eta V(\bbL) \cap D\,:\, \exists v \in \eta\bbL\,:\, u \sim v \textrm{ and }  v \in \eta\bbL \cap (\CC \setminus D)\} .
\]
A boundary condition $\xi$ is a partition of the discrete boundary of $D$.
A set in this partition denotes the vertices which are connected via red hexagons or edges
(depending on the model) in $\CC \setminus D$.
When $\xi$ is omitted, it means we are considering the full plane model and are not specifying any boundary conditions
on the discrete boundary of $D$.

\begin{assu}[Domain Markov Property]\label{assu:markov}
 Let $D\subset E\subset \CC$ be open sets. Further let $S\subset \overline{E\setminus D}$ and $T\subset \overline{D}$
 closed sets. Then
 \[
  \PP_{\eta}(\sigma_{D} = T\cap D \,|\, \sigma_{\overline{E\setminus D}} = S) = \PP_{\eta}(\sigma_{D} = T \,|\, \xi) =: \PP_{\eta}^{\xi}(\sigma_{D} = T)
 \]
where $\sigma_D = \sigma_{\eta} \cap D$ and $\xi$ is the discrete boundary condition on $D$ induced by $\sigma_{\overline{E\setminus D}} = S$.
\end{assu}

For some models the randomness is on the vertices (e.g. Bernoulli site percolation)
and for others on the edges (e.g. FK-Ising percolation). For the models
of the first form we define $\Omega_{\eta,D} := \eta V(\bbL)\cap D$ and for models
of the second form $\Omega_{\eta,D} := (\eta E(\bbL))|_D$. 

\begin{assu}[Strong positive association / FKG]\label{assu:strongPosAssoc}
 The finite measures are strongly positively-associated.
 More precisely, let $D\subset \CC$ be a bounded closed set. For every boundary condition $\xi$ on $\partial_{\eta}D$
 and increasing functions $f,g: \{\textrm{red}, \textrm{blue}\}^{\Omega_{\eta,D}} \to \RR$, we have
 \[
  \EE_{\eta}^{\xi}[f\cdot g] \ge \EE_{\eta}^{\xi}[f]\cdot \EE_{\eta}^{\xi}[g].
 \]
 Hence for increasing events $A, B$ and boundary condition $\xi$ on $\partial_{\eta}D$:
 \[
  \PP_{\eta}^{\xi}(A\cap B) \ge \PP_{\eta}^{\xi}(A)\PP_{\eta}^{\xi}(B).
 \]
\end{assu}
It is well known that monotonicity in the boundary condition
is equivalent to strongly positively-association, if the measure
is strictly positive (has the finite energy property), i.e. every configuration has
strictly positive probability. (See e.g. \cite[Theorem 2.24]{G06}.)
Furthermore it is well known that positive association survives the limit as the lattice grows towards infinity.
See for example \cite[Proposition 4.10]{G06}.

In the following assumption $l(Q)$ denotes the extremal length of $Q$,
that is, let $\phi: Q \to [0,a]\times [0,1]$ conformal such that $\phi(Q(\{0\} \times [0,1])) = \{0\} \times [0,1]$
and $\phi(Q(\{1\} \times [0,1])) = \{a\} \times [0,1]$, then $l(Q) = a$.
\begin{assu}[RSW]\label{assu:RSW}
 Let $M > 0$. There exist $\delta > 0$ such that, for every quad $Q$ with $l(Q) \le M$ and every boundary
 condition $\xi$ on the discrete boundary of $Q([0,1]^2)$:
 \[
  \PP_{\eta}^{\xi}(\omega_{\eta} \in \boxminus_{Q}) \ge \delta
 \]
 and for every quad $Q$ with $l(Q) \ge M$ and every boundary
 condition $\xi$ on the discrete boundary of $Q([0,1]^2)$:
 \[
  \PP_{\eta}^{\xi}(\omega_{\eta} \not\in \boxminus_{Q}) \le 1-\delta.
 \]
\end{assu}

\begin{assu}[Full Scaling Limit]\label{assu:existence_full_confInvLimit}
 As $\eta\rightarrow 0$, the law of $L_\eta$ converges weakly to a random infinite collection of loops
 $L_0$ in the induced Hausdorff metric on collections of loops induced by
 the distance \eqref{eq:def_loopdist} (similar to the metric $dist$
 defined in \eqref{eq:def_coll_dist_hat_CC}).
 Moreover, the limiting law is conformally invariant.
\end{assu}

\subsection{\label{sub:arm_events}Arm events}

For $S\subset\RS,$ let $\partial S,int\left(S\right),$ $\bar{S}$
denote the boundary, interior and the closure of $S$, respectively. We call the elements of $\{0,1\}^k$, $k\geq0$ as colour-sequences. For ease of
notation, we omit the commas in the notation of the colour sequences, e.g. we write $(101)$ for $(1,0,1)$.
\begin{definition} \label{def:arm_events}
  Let $l\in\mathbb{N},$ $\kappa\in\left\{ 0,1\right\} ^{l},$ $S\subseteq\RS$ and $D,E$ be two disjoint open,
  simply connected subsets
  of $\RS$ with piecewise smooth boundary. Let $D\xleftrightarrow{\kappa,S}E$ denote the
  event that there are $\delta>0$ and quads $Q_{i}\in\calQ_{S}$, $i=1,2,\ldots,l$ which
  satisfy the following conditions.
  \begin{enumerate}
    \item $\omega\in\boxminus_{Q_{i}}$ for $i\in\left\{ 1,2,\ldots,l\right\} $ with $\kappa_{i}=1$
    and $\omega\in\boxminus_{Q_{i}}^{c}$ for $i\in\left\{ 1,2,\ldots,l\right\}$ with $\kappa_{i}=0$.
    \item For all $i\neq j\in\left\{ 1,2,\ldots,l\right\} $ with $\kappa_{i}=\kappa_{j},$       the quads  $Q_{i}$ and $Q_{j}$, viewed as subsets of $\RS$,
    are disjoint, and are at distance at least $\delta$ from each other and from the boundary of $S$;
    \item $\Ball_\delta + Q_{i}\left(\left\{ 0\right\} \times\left[0,1\right]\right)\subset D$
    and $\Ball_\delta + Q_{i}\left(\left\{ 1\right\} \times\left[0,1\right]\right)\subset E$ for $i\in\left\{ 1,2,\ldots,l\right\} $ with $\kappa_{i}=1;$ 
    \item $\Ball_\delta + Q_{i}\left(\left[0,1\right]\times\left\{ 0\right\} \right)\subset D$
    and $\Ball_\delta + Q_{i}\left(\left[0,1\right]\times\left\{ 1\right\} \right)\subset E$ for $i\in\left\{ 1,2,\ldots,l\right\} $ with $\kappa_{i}=0;$
    \item The intersections $Q_{i} \cap D$, for $i = 1,2,\ldots, l$, are at distance at least $\delta$ from each other, the same holds for $Q_i \cap E$;
    \item A counterclockwise order of the quads $Q_{i}$ $i=1,2,\ldots,l$ is given by ordering counterclockwise
    the connected components of $Q_{i} \cap D$ containing $Q_{i}(0,0)$.
  \end{enumerate}
  When the subscript $S$ is omitted, it is assumed to be $\RS$.
\end{definition}

\begin{rem}\label{rem:MeasurabilityArmEvents} It is a simple exercise to show that the events $D\xleftrightarrow{\kappa,S}E$ are
$Borel(\mathscr{T}_{\hat\CC})$-measurable. See \cite[Lemma 2.9]{Garban2010} for more details.
\end{rem}

\medskip
In what follows we consider some special arm events. For $z\in\mathbb{C}, a > 0$ let $H_1(z,a),$ $H_2(z,a),$ $H_3(z,a),$ $H_4(z,a)$
denote the left, lower, right, and upper half planes which have the right, top, left and bottom
sides of $\Ball_a(z)$ on their boundary, \resp. For $z\in\CC,$ $0<a<b$ we set
\begin{align*}
  A(z;a,b) := \Ball_b(z)\setminus \Ball_a(z).
\end{align*}
Furthermore, for $i=1,2,3,4$, $\kappa\in \{0,1\}^l$ and $\kappa'\in \{0,1\}^{l'}$ with $l,l'\geq0$ we
define the event where there are $l+l'$ disjoint arms with
colour-sequence $\kappa\vee\kappa' := (\kappa_1,\ldots,\kappa_l,\kappa'_1,\ldots,\kappa'_{l'})$
in $A(z;a,b)$ so that the $l'$ arms, with colour-sequence $\kappa'$, are in the half-plane
$H_i(z,a)$. That is,
\begin{multline}
    \mathcal{A}^i_{\kappa,\kappa'}\left(z;a,b\right) :=\\
    \begin{aligned}
     & \left\{ \Ball_a(z)\xleftrightarrow{\kappa\vee\kappa'}\left(\RS\setminus \Ball_b(z)\right)\right\} \cap \left\{ \Ball_a(z)\xleftrightarrow{\kappa', H_i(z,a)}\left(\RS\setminus \Ball_b(z)\right)\right\} \label{eq:def_arm_event}
    \end{aligned}
\end{multline}
In the notation above, when $z$ is omitted, it is assumed to be $0$.
When $\kappa'=\emptyset$, both the subscript $\kappa'$ and the superscript $i$ will typically be omitted.
\begin{figure}
	\centering 
	\includegraphics[width = 0.60\textwidth]{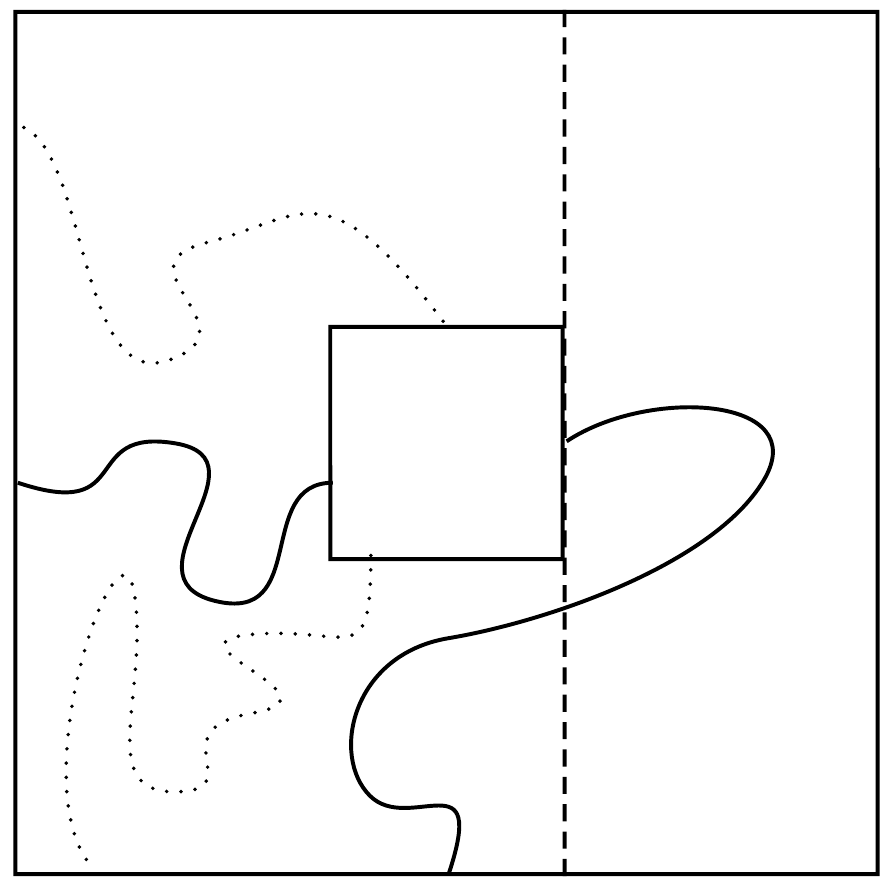}
	\caption{Illustration of the event $\mathcal{A}_{(1),(010)}^1(a,b)$.}
	\label{fig:armEvent}
\end{figure}
See Figure \ref{fig:armEvent} for an illustration of an arm event.

Finally, for $0<a<b$ and boundary condition $\xi$ on $\partial_{\eta}\Ball_{b}$ we set
\begin{align*}
  \pi_1^{\eta,\xi}(a,b) &:= \PP_\eta^{\xi}(\mathcal{A}_{(1)}\left(a,b\right)),  & \pi_{4}^{\eta,\xi}(a,b) &:= \PP_\eta^{\xi}(\mathcal{A}_{(1010),\emptyset}^{1}\left(a,b\right)),\\
  \pi_6^{\eta,\xi}(a,b) &:= \PP_\eta^{\xi}(\mathcal{A}_{(010101),\emptyset}^{1}\left(a,b\right)), & \pi_{0,3}^{\eta,\xi}(a,b) &:= \PP_\eta^{\xi}(\mathcal{A}_{\emptyset,(010)}^{1}\left(a,b\right)),\\
  \pi_{1,3}^{\eta,\xi}(a,b) &:= \PP_\eta^{\xi}(\mathcal{A}_{(1),(010)}^{1}\left(a,b\right)).
\end{align*}

\begin{rem}
 The (technical) reason to define $H_i(z,a)$ in this slightly unnatural way will become clear
 in the proof of Lemma \ref{lem:good_sub}.
\end{rem}

\subsection{Consequences of RSW}

\begin{lemma}[Quasi multiplicativity] \label{assu:q_mult}
 Suppose that Assumptions \ref{assu:markov}-\ref{assu:RSW} hold.
  There is a constant $C >0$ such that
  \begin{align*}
    \PP_\eta^{\xi}(\mathcal{A}_{(1)}\left(a,b\right)) \leq C \frac{\pi_1^{\eta,\xi}(a,c)}{\pi_1^{\eta,\xi}(b,c)}
  \end{align*}
  for all $a,b,c,\eta >0$ with $\eta < a < b < c$ and boundary condition $\xi$ on $\partial_{\eta}\Ball_{c}$.
\end{lemma}

\begin{lemma} \label{assu:1arm_exp_ubound}
 Suppose that Assumptions \ref{assu:markov}-\ref{assu:RSW} hold.
  There are constants $\lambda_1\in (0,1)$ and $C>0$ such that
  \begin{align*}
    \pi_{1}^{\eta,\xi}(\eta,b) \geq C\left(\frac{a}{b}\right)^{\lambda_1}
	\PP_\eta^{\xi}(\mathcal{A}_{(1)}\left(\eta,a\right))
  \end{align*}
  for all $b>a>\eta$ and boundary condition $\xi$ on $\partial_{\eta}\Ball_{b}$.
\end{lemma}

\begin{lemma} \label{assu:arm_exp}
 Suppose that Assumptions \ref{assu:markov}-\ref{assu:RSW} hold.
  There are positive constants $C,\lambda_6$ such that
\begin{align}
  \pi_6^{\eta,\xi}(a,b) &\leq C\left(\frac{a}{b}\right)^{2+\lambda_6}, & \pi_{0,3}^{\eta,\xi}(a,b) &\leq C\left(\frac{a}{b}\right)^{2}
\end{align}
for all $0<\eta<a<b$ and boundary condition $\xi$ on $\partial_{\eta}\Ball_{b}$.
\end{lemma}

\begin{lemma} \label{assu:3hp_plusone}
 Suppose that Assumptions \ref{assu:markov}-\ref{assu:RSW} hold.
  There are positive constants $C,\lambda_{1,3}$ such that
\begin{equation}
  \pi_{1,3}^{\eta,\xi}(a,b) \leq C\left(\frac{a}{b}\right)^{2+\lambda_{1,3}}
\end{equation}
for all $0<\eta<a<b$ and boundary condition $\xi$ on $\partial_{\eta}\Ball_{b}$.
\end{lemma}

\begin{lemma} \label{assu:1arm_exp_less_4arm}
 Suppose that Assumptions \ref{assu:markov}-\ref{assu:RSW} hold.
 There are constants $C,\lambda>0$ such that
  \begin{align*}
     \frac{\pi_{1}^{\eta,\xi}(a,b)}{\pi_{4}^{\eta,\xi}(a,b)}\geq C \left(\frac{b}{a}\right)^{\lambda}
  \end{align*}
  for all $b>a>\eta$ and boundary condition $\xi$ on $\partial_{\eta}\Ball_{b}$.
\end{lemma}

For the sake of generality, we have stated the bounds in the previous lemmas in the presence of boundary conditions.
However, in the rest of the paper only the full-plane versions of the bounds will appear, so the superscript $\xi$ will be dropped.
(The versions with boundary conditions are necessary to obtain results that we use in this paper, but whose proofs we do not reproduce.)
For the next lemma we need some additional notation.
\begin{definition}\label{def:V_a}
  For $\eta,a>0$ let
  \begin{align*}
    \calV_a^\eta := \{v\in \Ball_{a/2} \cap \eta V \,|\, v\xleftrightarrow{1} \partial\Ball_{a} \text{ in } \omega_\eta\}
  \end{align*}
  denote the number of vertices in $\Ball_{a/2}$ connected to $\partial \Ball_{a}$ in $\sigma_\eta$.
\end{definition}
\begin{lemma} \label{lem:BCKS}
  Suppose that Assumptions \ref{assu:markov}-\ref{assu:RSW} hold. Then there are positive constants $c,C$ such that
  \begin{align*}
    \PP_{\eta}(|\calV_a^\eta| \geq x(a/\eta)^2\pi_1^\eta(\eta,a)) \leq Ce^{-cx}
  \end{align*}
  for all $a > \eta$ and $x\geq0$.
\end{lemma}

\begin{lemma}\label{lem:L3bound}
 Suppose that Assumptions \ref{assu:markov}-\ref{assu:RSW} hold. Then there is a constant $C>0$ such that
  \begin{align*}
    \EE_{\eta}[|\calW_a^\eta|^{3}] \leq C\eta^{-6}\pi_1^\eta(\eta,a)^{3}
  \end{align*}
  for all $0 < \eta < a < 1/2$, where
  \begin{align*}
    \calW_a^\eta := \{v\in \Ball_{1} \cap \eta V \,|\, v\xleftrightarrow{1} \partial\Ball_{a}(v) \text{ in } \omega_\eta\}.
  \end{align*}
\end{lemma}

\begin{proof}[Proof of Lemmas \ref{assu:q_mult} - \ref{lem:L3bound}]
 Lemmas \ref{assu:arm_exp} and \ref{assu:3hp_plusone} follow from Assumptions \ref{assu:markov} - \ref{assu:RSW},
 as explained in e.g. \cite{Nolin2008,Grimmett1999} for the case of Bernoulli percolation and in
\cite[Corollary 1.5 and Remark 1.6]{CD-CH2013} for the case of FK-Ising percolation.
 (The additional boundary conditions, which are not present in the above mentioned corollary
 and remark in \cite{CD-CH2013},
 do not affect the results. This can easily be deduced from
 equation (5.1) in \cite{CD-CH2013}.)

Also Lemmas \ref{assu:1arm_exp_ubound} and \ref{assu:1arm_exp_less_4arm} follow from standard RSW, FKG arguments.

Lemma \ref{assu:q_mult} is similar to \cite[Theorem 1.3]{CD-CH2013},
which is shown to follow from our assumptions \ref{assu:markov}-\ref{assu:RSW}.
The boundary condition on $\partial_\eta \Ball_c$ has no effect on the proof, because
the RSW result is uniform in the boundary conditions. (Furthermore there is no need to ``make" the arms
well separated on $\partial_\eta \Ball_c$.)

An easy proof of Lemma \ref{lem:BCKS} for critical percolation can be found in \cite{N88}.
It is easy to see that the same proof can be modified in such a way that
the result follows from Lemmas \ref{assu:q_mult} - \ref{assu:1arm_exp_less_4arm},
and hence from Assumptions \ref{assu:markov}-\ref{assu:RSW}.
For percolation, Lemma \ref{lem:BCKS} can also be found in \cite[Lemma 6.1]{BCKS99}, and for 
FK-Ising percolation in \cite[Lemma 3.10]{CGN15}.

Finally Lemma \ref{lem:L3bound} can be proved easily using Lemma \ref{assu:q_mult}.
See for example \cite[Lemma 4.5]{Garban2010} or the proof of Lemma \ref{lem:BCKS}.
\end{proof}

\subsection{Additional preliminaries}

\begin{lemma}\label{lem:convergenceQuadCrossings}
 Suppose that Assumptions \ref{assu:markov}-\ref{assu:existence_full_confInvLimit} hold.
 The set of crossed quads is, almost surely, measurable with respect to the collection of loops.
\end{lemma}

\begin{proof}[Proof of Lemma \ref{lem:convergenceQuadCrossings}]
A proof of this can be found in \cite[Section 2.3]{Garban2010} and follows
almost immediately from arguments given in \cite[Section 5.2]{Camia2006}.
The proof of the measurability of quad crossings with respect to the collection of 
loops makes use of three properties of the loop process, which all follow
from RSW techniques (see the first three items of
Theorem 3 in \cite[Section 5.2]{Camia2006}).
Because of this, the measurability is a simple consequence of
our Assumptions \ref{assu:markov}-\ref{assu:existence_full_confInvLimit}.
\end{proof}

\begin{rem}
  Assumption \ref{assu:existence_full_confInvLimit}, together with the separability of $\mathscr{H}_{\RS}$, implies that there is a coupling $\PP$ so that $\omega_\eta\rightarrow\omega_0$ a.s. as
  $\eta\rightarrow0$.
\end{rem}

Before we proceed to the next lemma, we recall the following result on the scaling limits of arm events. A slightly weaker version of the following lemma appeared as \cite[Lemma 2.9]{Garban2010}. Its proof extends immediately to the more general case.

\begin{lemma}[Lemma 2.9 of \cite{Garban2010}] \label{lem:GPS}
  Suppose that Assumptions \ref{assu:markov}-\ref{assu:existence_full_confInvLimit} hold. Then, under a coupling
  $\PP$ of $(\PP_{\eta})_{\eta \ge 0}$ such that $\omega_{\eta} \to \omega_{0}$ almost surely,
  we have for events $\calD \in \{\{A\xleftrightarrow{(1),S} B\}, \{A\xleftrightarrow{(010),S} B\}, \mathcal{A}^i_{\kappa,\kappa'}\left(z;a,b\right)\}$,
  \begin{align*}
   \ind_{\calD}(\omega_{\eta}) \to \ind_{\calD}(\omega_{0}) \qquad \textrm{in }\PP\textrm{-probability},
  \end{align*}
  for $(\kappa,\kappa') \in \{ ((1),\emptyset),((1010),\emptyset),((010101),\emptyset),(\emptyset,(010)),((1),(010)) \}$,
  rectangle $S \subseteq \CC$, $i\in\{1,2,3,4\}$, $0<a<b$ and $A,B$ disjoint open subsets of $\CC$ with
  piecewise smooth boundary.
\end{lemma}

The lemma above implies that for all $a,b>0$ with $a<b$ the probability $\pi_1^\eta(a,b)$ converges as $\eta\ra 0$. We write $\pi^0_1(a,b)$ for the limit. General arguments \cite[Section 4]{beffara2011} using Lemma \ref{assu:q_mult} above show that
\begin{align} \label{eq:1arm_exp}
  \pi^0_1(a,b) = \left(\frac{a}{b}\right)^{\alpha_1+o(1)}
\end{align}
for some $\alpha_1\geq0$ where $o(1)$ is understood as $b/a \ra \infty$. Lemma \ref{assu:1arm_exp_ubound} shows that $\alpha_1 < 1 $.

We need some additional notation for the next theorems.
For $z\in\CC$ and $a>0$ let $\Ball'_a(z) := \{u\in \CC \,|\, \Re(u-z),\Im(u-z)\in[-a,a) \}.$ Note that $\Ball_a(z)$ and $\Ball'_a(z)$ differ only on their boundary.
For an annulus $A = A(z;a,b)$ let
\begin{align} \label{eq:def_1armMeas}
  \mu_{1,A}^{\eta} := \frac{\eta^2}{\pi_1^{\eta}(\eta,1)}\sum_{v\in \Ball'_a(z)\cap \eta V}\delta_v
  \ind_{\{v\leftrightarrow \partial \Ball_b(z) \text{ in } \omega_{\eta}\}}
\end{align}
denote the counting measure of the vertices in $\Ball'_a(z)$ with an arm to $\partial \Ball_b(z)$ at scale $\eta$.

\begin{thm} \label{assu:conv_measures}
  Suppose that Assumptions \ref{assu:markov}-\ref{assu:existence_full_confInvLimit} hold. Let $A = A(z;a,b)$ be an annulus, and $\PP$ be a coupling
  such that $\omega_\eta\ra\omega_0$ a.s. as $\eta\ra0$. Then the measures $\mu_{1,A}^{\eta}$ converge
  weakly to $\mu_{1,A}^{0}$ in probability under the coupling $\PP$ as $\eta$ tends to $0$. Furthermore, $\mu_{1,A}^0$ is a measurable function of $\omega_{0}$. In particular,
  the pair $(\omega_\eta,\mu_{1,A}^\eta)$ converges to $(\omega_0,\mu_{1,A}^0)$ in distribution as $\eta\ra0$.
\end{thm}

Theorem \ref{assu:conv_measures} is proved for site percolation on the triangular lattice
in \cite{Garban2010} where it is Theorem 5.1.
Namely, it is easy to check that the proof of \cite[Theorem 5.1]{Garban2010} shows that the measures
$\mu_{1,A}^\eta\xrightarrow{p}\mu_{1,A}^0$ under the coupling $\PP$ converge weakly in probability as $\eta\ra0$.
For FK-Ising, a sketch proof for a theorem similar to this was given in \cite{CGN15}.
Unfortunately the proof contains a mistake, but luckily the mistake can be easily fixed.
Below we give an informal sketch of the proof of Theorem \ref{assu:conv_measures},
following the proof in \cite{CGN15} and briefly explaining how to fix it.

The strategy is to approximate, in the $L^2$-sense, the one-arm measure by the number of mesoscopic boxes
connected to $\partial \Ball_b(z)$, multiplied by a constant depending on the size of the boxes.
Here mesoscopic means much larger than the mesh size $\eta$ but much smaller than $a$.

In order to get $L^2$-bounds on the error terms, first we use a coupling argument to argue that the boxes which
are far away from each other are almost independent. Namely, with high probability one can draw 
a red circuit around one of the boxes, which is also conditioned on having a long
red arm (because of positive association, that event can only increase the probability of a red circuit).
This red circuit makes, via the Domain Markov Property, the contribution of the surrounded box independent of that of the other boxes.
The total contribution of the boxes which are close to each other is negligible.
Secondly we use a ratio limit argument, based on the existence of the one-arm exponent $\alpha_1$
from \eqref{eq:1arm_exp}, to show that the contribution of a
single box is approximately a constant, which only depends on the size of the mesoscopic box.

The small mistake in \cite{CGN15} mentioned above is in the assumption that the convergence in
Lemma \ref{lem:GPS} is almost sure, as claimed in an earlier version of \cite{Garban2010}. However, as noted
in the final version of \cite{Garban2010}, one can only prove convergence in probability. Luckily, arguments in \cite{Garban2010}
show that
convergence in probability, together with $L^3$ bounds from Lemma \ref{lem:L3bound},
is sufficient to prove convergence in $L^2$ of the number of mesoscopic boxes connected to $\partial \Ball_b(z)$ times
a constant depending on the size of these boxes.

\subsection{Validity of the assumptions}
\subsubsection{The case of critical percolation} \label{ssec:crit_perc_cond}
Now we check that the Assumptions above hold for critical site percolation on the triangular lattice.
\begin{thm}\label{thm:assu_perco}
  For critical site percolation on the triangular lattice, the Assumptions \ref{assu:markov}-\ref{assu:existence_full_confInvLimit} hold.
\end{thm}

\begin{proof}[Proof of Theorem \ref{thm:assu_perco}]
 The Domain Markov Property, Assumption \ref{assu:markov}, is trivial
 for Bernoulli percolation. 
 Assumption \ref{assu:strongPosAssoc} is well known (see, e.g., \cite[Theorem 3.8]{G06}).
 RSW, Assumption \ref{assu:RSW}, is also well known (see, for example, \cite{Grimmett1999,Nolin2008}). 
 The existence of the full scaling limit in Assumption \ref{assu:existence_full_confInvLimit} is proved
 by the first author and Newman in \cite{Camia2006}.
 We note that the value of $\alpha_1$ for Bernoulli percolation is $5/48$, as proved in \cite{LSW02}.
\end{proof}

\subsubsection{The case of the critical FK-Ising model} \label{ssec:FK_cond}
The Domain Markov Property and strong positive association are standard and well known (see, e.g., \cite{G06}).
The recent development of the RSW theory for the FK-Ising model proves Assumption \ref{assu:RSW}. Namely,
Assumption \ref{assu:RSW} follows from Theorem 1.1 in \cite{CD-CH2013} combined with the fact that the
discrete extremal length, used in \cite{CD-CH2013}, is comparable to its continuous counterpart, used here
(see \cite[Proposition 6.2]{Chelkak12}).
Recently, a proof of the uniqueness of the full scaling limit for the critical FK-Ising model has
been completed by Kemppainen and Smirnov \cite{KS16}. Theorem 1.7 in their paper implies
Assumption \ref{assu:existence_full_confInvLimit}.
We note that the value of $\alpha_1$ for the Ising model is $1/8$. As shown in \cite{CN09},
this can be seen from the behavior of the Ising two-point function at criticality \cite{W66}. 

\section{Approximations of large clusters} \label{sec:approx}

In what follows we give two approximations of open clusters with diameter at least $\delta > 0$,
which are completely contained in $\Ball_k$.
The first one relies solely on the arm events described in the previous section, while the other is `the natural' one, namely it is simply the union of $\varepsilon$-boxes
which intersect the cluster.  The advantage
of the first approximation is that it can also be defined in the limit as the mesh size goes to $0$.
First we prove Proposition \ref{prop:good_sub} below, which shows that, on a
certain event, these two approximations coincide. Then in Section \ref{ssec:error_bounds} we give a lower bound for the probability of that event.

For simplicity, we set $k=1$ from now on. The constructions and proofs for different values of $k$ are analogous. 
Let $\ZZ[\imag] = \{a+b\imag\,|\, a,b\in\ZZ\}$. For $\varepsilon > 0,$ let $B_\varepsilon$ be the following collection of squares of side length $\varepsilon$:
\begin{align*}
  B_\varepsilon := \left\{\Ball_{\varepsilon/2}(\varepsilon z) \,|\, z \in \Ball_{\lceil 1/\varepsilon\rceil}\cap \ZZ[\imag]\right\}.
\end{align*}

Fix $\omega \in \mathscr{H}_{\RS}$. We define the graph $G_\varepsilon = G_\varepsilon(\omega)$ as follows. Its vertex set is $B_\varepsilon$.
The boxes $\Ball_{\varepsilon/2}(\varepsilon z),\Ball_{\varepsilon/2}(\varepsilon z')\in B_\ve$ are connected by an edge if
$||z-z'||_\infty = 1$ or if $\omega \in \{\Ball_{\varepsilon/2}(\varepsilon z)\xleftrightarrow{(1)}\Ball_{\varepsilon/2}(\varepsilon z')\}$.
For a graph $H$ with $V(H)\subseteq B_\varepsilon$ we set 
\begin{align}\label{eq:def_U}
  U(H) := \bigcup_{\Ball\in V(H)}\Ball \subseteq \Ball_{1+2\varepsilon}.
\end{align}
Let $L(H)$ denote the set of leftmost vertices of $H$. That is,
\begin{align*}
  L(H) := \{\Ball_{\varepsilon/2}(\varepsilon z) \in V(H) \, |\, \forall z'\in \ZZ[\imag] \text{ with } \Ball_{\varepsilon/2}(\varepsilon z') \in V(H), \Re z \leq \Re z'\}.
\end{align*}
Similarly, we define $R(H), T(H), B(H)$ as the rightmost, top and bottom sets of vertices of $H$, respectively.
Let $SH(H)$ (resp. $SV(H)$) denote the narrowest double-infinite horizontal (resp. vertical) strip containing $U(H)$.
Finally, let $SR(H)$ denote the smallest rectangle containing $U(H)$ with sides parallel to one of the axes.
Thus $SR(H) = SH(H) \cap SV(H)$.

\begin{definition}
  For $z,z'\in\CC$, we set $\dist_1(z,z') = |\Re(z-z')|$ and $\dist_2(z,z') = |\Im(z-z')|$.
  We call $\dist_1$ (resp. $\dist_2$) the distance in the horizontal (resp. vertical) direction.
  We also use the notation $d_\infty(z,z'):=||z-z'||_\infty = \dist_1(z,z')\vee \dist_2(z,z')$ for the $L^\infty$ 
  distance.
  
  For disjoint sets $A,B \subset \RS$ we set $\dist_{i}(A,B) := \inf \{\dist_{i}(z,z'): z \in A, z' \in B \}$ for $i=1,2$.
\end{definition}

Let $\eta > 0$, $\Ball = \Ball_{\varepsilon/2}(z)\in B_\varepsilon$ and $\Ball' = \Ball_{\varepsilon/2}(z')\in B_\varepsilon$.
Suppose there is a cluster which is completely contained in $\Ball_1$, such that
$\Ball$ contains a leftmost vertex of this cluster and $\Ball'$ a rightmost vertex.
Then $\Ball$ and $\Ball'$ are connected by 2 blue arms and one red arm in between them.

This leads us to the following definition, which gives us a way to characterize the clusters using only arm events.
\begin{definition} \label{def:good_sub}
  Let $\omega \in \mathscr{H}_{\RS}$ and $G_\varepsilon = G_\varepsilon(\omega)$ the graph defined above.
  Let $H$ be a subgraph of $G_\varepsilon(\omega)$. We say that $H$ is good, if it satisfies the following conditions:
  \begin{enumerate}
    \item $H$ is complete,
    \item \label{cond:contained} $U(H) \subseteq \Ball_{1}$,
    \item $H$ is \label{cond:max} maximal, that is, if $\Ball \in V(G_\varepsilon)$ and $(\Ball,\Ball') \in E(G_\varepsilon)$ for all $\Ball' \in V(H)$, then $\Ball\in V(H)$,
    \item \label{cond:diam} $\diam(U(H)) \geq \delta$,
    \item for all $\Ball \in L(H)$ and $\Ball' \in R(H)$ we have $\omega \in \{\Ball\xleftrightarrow{(010), SV(H)} \Ball'\}$,
    a similar condition holds for $\Lambda \in T(H)$ and $\Lambda'\in B(H)$, with $SV(H)$ replaced by $SH(H)$.
  \end{enumerate}
\end{definition}

For a set $S\subseteq \CC$ and $\varepsilon > 0$ let $K_\varepsilon(S)$ denote the complete graph on the vertex set
\begin{align*}
  \left\{\Ball_{\varepsilon/2}(\varepsilon z) \,|\, z\in\ZZ[\imag] \text{ and } \Ball_{\varepsilon/2}(\varepsilon z)\cap S \neq \emptyset \right\}.
\end{align*}
Further, we introduce the shorthand notation
\begin{align*}
  U_\varepsilon(S) := U(K_\varepsilon(S)) = \bigcup_{z\in \ZZ[\imag]: \Ball_{\varepsilon/2}(\varepsilon z)\cap S\neq \emptyset} \Ball_{\varepsilon/2}(\varepsilon z).
\end{align*}
For $\cl_\eta \in \clColl_1^{\eta}(\delta)$, the graph $K_{\varepsilon}(\cl_{\eta})$ approximates $\cl_{\eta}$
in the sense that $d_{H}(\cl_{\eta}, U_{\varepsilon}(\cl_{\eta})) < \varepsilon$. This is the second approximation
of large clusters we referred to in the beginning of this section. Our next aim is to find an event where the two
approximations coincide.

In what follows we use the quantities defined above in the case where $\omega = \omega_\eta$ for some $\eta\geq 0$.
We denote the particular choice of $\eta$ in the superscript, for example $G^\eta_\varepsilon := G_\varepsilon(\omega_\eta)$. We shall prove:

\begin{prop}\label{prop:good_sub}
  Let $\eta,\varepsilon,\delta>0$ with $1/10 > \delta > 10\varepsilon$. Suppose that $\omega_\eta\in \calE(\varepsilon,\delta)$, where $\calE(\varepsilon,\delta)$ is defined in \eqref{eq:def_E} below.
  \begin{enumerate}[i)]
	\item Then for each good subgraph $H$ of $G_\varepsilon^\eta$ there is a unique cluster $\cl^\eta\in\clColl_1^{\eta}(\delta)$
  such that $H = K_\varepsilon(\cl^\eta)$.
	\item Conversely, if $\cl^\eta\in\clColl_1^{\eta}(\delta)$, then $K_\varepsilon(\cl^\eta)$ is a good subgraph of $G_\varepsilon^\eta$.
  \end{enumerate}
\end{prop}
\begin{proof}[Proof of Proposition \ref{prop:good_sub}]
   Proposition \ref{prop:good_sub} follows from the combination of Lemmas \ref{lem:diam_almost_delta} and
   \ref{lem:good_sub} with the definition \eqref{eq:def_E} below.
\end{proof}

For $\varepsilon, \delta>0$ we define the event
\newcommand{\NA}{\mathcal{NA}}
\newcommand{\noClu}{\mathcal{NC}}
\begin{align}
  \calE(\varepsilon,\delta) := \NA\left(\varepsilon,\delta\right) \cap \noClu(\varepsilon,\delta). \label{eq:def_E}
\end{align}
First we define the event $\noClu(\varepsilon,\delta)$ below, then we introduce $\NA(\varepsilon,\delta)$ in Definition \ref{def:NA}.

\begin{definition} \label{def:def_NC}
  Let  $0<10\varepsilon < \delta < 1$. We write $\noClu(\varepsilon, \delta)^c$ for the union of events
  \begin{align}\label{eq:cl_diam_almost_delta2}
    \calA_{\emptyset,(010)}^{j}(z;\varepsilon/2,\delta/2-3\varepsilon) \cap \calA_{\emptyset,(010)}^{j+2}(z';\varepsilon/2,\delta/2-3\varepsilon)
  \end{align}
  for $j=1,2$, and $z,z' \in \Ball_{\lceil 1/\varepsilon\rceil}\cap \ZZ[\imag]$
  with $\dist_j(z,z') \in (\delta - 3\varepsilon, \delta + 3\varepsilon)$.
\end{definition}
Definition \ref{def:def_NC} implies the following lemma, which explains the choice of the
event $\noClu(\varepsilon, \delta)$.
\begin{lemma} \label{lem:diam_almost_delta}
  Let $0 < 10\varepsilon<\delta<1$. On $\omega_\eta\in\noClu(\varepsilon, \delta)$ there is no cluster $\cl^\eta$,
  which is completely contained in $\Ball_1$ with
  diameter between $\delta-2\varepsilon$ and $\delta$.
\end{lemma}

We define the event $\NA(\varepsilon, \delta)$ which will be crucial in what follows.
\begin{definition} \label{def:NA}
Let $\varepsilon,\delta$ with $0<10\varepsilon<\delta<1$. We set $\NA_1\left(\varepsilon,\delta\right)$ for the complement of the event
\[
  \bigcup_{z\in \Ball_{\lceil 1/\varepsilon\rceil}\cap \ZZ[\imag]}\bigcup_{j=1}^4 \calA^j_{1,(010)}(\varepsilon z;\varepsilon/2,\delta/2 - 3\varepsilon).
\]
We write $\NA_2(\varepsilon, \delta)^c$ for the union of events
  \begin{align}\label{eq:cl_diam_almost_delta}
    \calA_{\emptyset,(010)}^{j}(z;\varepsilon/2,\delta/2-3\varepsilon)
  \end{align}
  for $j=1,2,3,4$, and $z \in \Ball_{\lceil 1/\varepsilon\rceil}\cap \ZZ[\imag]$
  with $\min_{i\in \{1,2\}}\dist_{i}(\Ball_{\varepsilon/2}(z),\partial \Ball_1) \le \varepsilon$.
  We define $\NA(\varepsilon,\delta) := \NA_1(\varepsilon, \delta) \cap \NA_2(\varepsilon, \delta)$.
\end{definition}

\begin{lemma}\label{lem:good_sub}
  Let $\eta,\varepsilon,\delta> 0$ with $0 < 10\varepsilon < \delta < 1$ and suppose that $\omega_\eta\in \mathcal{NA}(\varepsilon,\delta)$.
    \begin{enumerate}[i)]
     \item \label{part:cluster_then_good} If $\cl^\eta\in\clColl_1^\eta(\delta),$ then $K_\varepsilon(\cl^\eta)$ is
     a good subgraph of $G_\varepsilon^\eta$.
    \item \label{part:good_then_cluster} Conversely, for any good subgraph $H$ of $G_\varepsilon^\eta$,
    there is a unique cluster $\cl^\eta\in\clColl_1^\eta(\delta -2\varepsilon)$ 
    such that $H = K_\varepsilon(\cl^\eta)$.
  \end{enumerate}
\end{lemma}

\begin{proof}[Proof of Lemma \ref{lem:good_sub}] 
Let $\varepsilon,\delta$ as in the lemma, and $\omega_\eta\in \NA(\varepsilon,\delta)$.
First we prove part \ref{part:cluster_then_good}) above. Apart from conditions
\eqref{cond:contained} and \eqref{cond:max}, the conditions in Definition \ref{def:good_sub} are trivially
satisfied. The fact that $\omega_{\eta} \in \NA_2(\varepsilon, \delta)$ implies that condition \eqref{cond:contained}
is satisfied.
We prove condition \eqref{cond:max} by contradiction.

Suppose that condition \eqref{cond:max} is violated. Then there is $\Ball\in V(G_\varepsilon^\eta)
\setminus V(K_\varepsilon(\cl^\eta))$ such that $(\Ball,\Ball') \in E(G_\varepsilon^\eta)$ for all
$\Ball'\in V(K_\varepsilon(\cl^\eta))$.

We can assume that the diameter of $\cl^\eta$ is realized in the horizontal direction.
Take $L\in L(K_\varepsilon(\cl^\eta))$ and $R\in R(K_\varepsilon(\cl^\eta))$. Let $\gamma$ denote a path in $\cl^\eta$
connecting $L$ and $R$. We can further assume that $\dist_{1}(\Ball, L) > \delta/2 - \varepsilon$. Note that $\gamma$ is not
connected to $\Ball$. However, $\Ball$ is connected to $L$. Hence the blue boundary of $\cl^\eta$
separates $\gamma$ from the connection between $\Ball$ and $L$.
We get, from $L$ to distance $\delta/2-\varepsilon$, three half plane arms with colour sequence $(010)$, and a fourth red arm from the
connection between $\Ball$ and $L$. In particular, $\omega_\eta\in \NA_1(\varepsilon, \delta)^c$,
 giving a contradiction and proving part \ref{part:cluster_then_good}) of Lemma \ref{lem:good_sub}.

\medskip
Now we proceed to the proof of part \ref{part:good_then_cluster}). We may assume that the diameter of $U(H)$ is realized between a leftmost and a rightmost point of it.
Let $L\in L(H)$, $R\in R(H)$ and $\gamma$ be a path in $SR(H)$ connecting $L$ and $R$.
Furthermore, let $\Ball'\in V(G_\varepsilon^\eta)$ be such that $\gamma$ is connected to $\Ball'$ by
a path in $\sigma_\eta \cap \Ball_{1}$.

We show that $(\Ball,\Ball')\in E(G_\varepsilon^\eta)$ for all
$\Ball\in V(H)$. Suppose the contrary, i.e. there is $\Ball\in V(H)$ such that $(\Ball,\Ball')\notin E(G_\varepsilon^\eta)$.
Then $\Ball$ is not connected to $\gamma$. Furthermore, we may assume that
$\dist_{1}(\Ball, L) > \delta/2 - \varepsilon$. Then as above, we find three half plane arms with colour
sequence $(010)$ and a fourth red arm starting at $L$ to distance $\delta/2 - \varepsilon$.
In particular, $\omega_\eta\in \mathcal{NA}_1(\varepsilon,\delta)^c$,
which contradicts the assumption on $\omega_\eta$ above.

Hence $\Ball'\in V(H)$ since $H$ is maximal. Thus $K_\varepsilon(\cl^\eta(\gamma))$ is a subgraph of $H$,
where $\cl^\eta(\gamma)$ denotes the connected component of $\gamma$ in $\sigma_\eta$.
Note that $K_{\varepsilon}(\cl^\eta(\gamma))$ is a good subgraph because it satisfies condition
\eqref{cond:diam}, since $\dist_{1}(L,R)>\delta$, and condition \eqref{cond:max}, by part \ref{part:cluster_then_good})
of Lemma \ref{lem:good_sub}.
This completes the proof of part \ref{part:good_then_cluster}) and that of Lemma \ref{lem:good_sub}.
\end{proof}

The proof above implies the following useful property of the event $\mathcal{NA}(\varepsilon,\delta)$.
\begin{lemma} \label{lem:bound_cluster_count}
    Let $\eta,\varepsilon,\delta >0$ with $0 < 10\varepsilon  < \delta< 1$. If  $\omega_\eta \in \mathcal{NA}(\varepsilon,\delta)$, then we have $|\clColl_1^\eta(\delta)| \le 32\varepsilon^{-2}$.
\end{lemma}
\begin{proof}[Proof of Lemma \ref{lem:bound_cluster_count}]
   Let $\cl,\cl'\in \clColl_1^\eta(\delta)$ be clusters with diameter at least $\delta$ in the horizontal direction.
   The proof of Lemma \ref{lem:good_sub} shows that  on the event $\NA(\varepsilon,\delta)$, $L(K_{\varepsilon}(\cl))$
   and $L(K_{\varepsilon}(\cl'))$ are disjoint. The same holds for pairs of clusters with vertical diameter at least $\delta$.
   Thus $|\clColl_1^\eta(\delta)| \leq 2 (2\lceil 1/\varepsilon\rceil)^2) \leq 32\varepsilon^{-2}$.
\end{proof}

\subsection{Bounds on the probability of the events $\noClu(\varepsilon,\delta)$ and $\NA(\varepsilon, \delta)$} \label{ssec:error_bounds}

Our aim in this section is to prove the following bound on the probability of the complement of $\calE(\varepsilon,\delta)$, defined in \eqref{eq:def_E}.

\begin{prop} \label{prop:ubound_calE}
  Let $\varepsilon,\delta$ with $0<10\varepsilon<\delta<1$. Suppose that Assumptions \ref{assu:markov}-\ref{assu:RSW} hold.
  Then there are positive constants $C=C(\delta),\lambda$ such that for all $\eta\in(0,\varepsilon)$ we have
  \begin{align*}
    \PP_\eta\left(\calE(\varepsilon,\delta)^{c}\right) \leq C\varepsilon^\lambda.
  \end{align*}
\end{prop}
The proof of the proposition above follows from Lemma \ref{lem:ubound_NA} and \ref{lem:ubound_NC} below.
We start with an upper bound on the probability of the complement of $\NA(\varepsilon, \delta)$.

\begin{lemma} \label{lem:ubound_NA}
  Suppose that Assumptions \ref{assu:markov}-\ref{assu:RSW} hold.
  Let $\varepsilon,\delta$ with $0<10\varepsilon<\delta<1$. Then there are constants $C = C(\delta),\lambda >0$ such that
\begin{align} \label{eq:bound_NA}
  \PP_\eta(\NA\left(\varepsilon,\delta\right)^c) \leq C\varepsilon^\lambda
\end{align} for all $\eta <\varepsilon$. 
In particular, $|\clColl^\eta_1(\delta)|$ is tight in $\eta$ for all fixed $\delta > 0$.
\end{lemma}
\begin{proof}[Proof of Lemma \ref{lem:ubound_NA}]
  For $\varepsilon,\delta$ with $0<10\varepsilon<\delta<1$ simple union bounds together with Lemmas \ref{assu:arm_exp} and \ref{assu:3hp_plusone} give
  \begin{align*}
    \PP_\eta(\NA_{1}\left(\varepsilon,\delta\right)^c) & \leq 10 \varepsilon^{-2}\left(\frac{\varepsilon}{\delta}\right)^{2+\lambda_{1,3}} = 10\frac{\varepsilon^{\lambda_{1,3}}}{\delta^{2+\lambda_{1,3}}},\\
    \PP_\eta(\NA_{2}\left(\varepsilon,\delta\right)^c) & \leq 40 \varepsilon^{-1}\left(\frac{\varepsilon}{\delta}\right)^{2}  = 40\frac{\varepsilon}{\delta^{2}}.
  \end{align*}
  This, combined with the definition of the event $\NA(\varepsilon,\delta)$, provides the desired upper bound.
  The tightness of $|\clColl^\eta_1(\delta)|$ follows from the combination of Lemma \ref{lem:bound_cluster_count} and \eqref{eq:bound_NA}.
\end{proof}

\begin{lemma} \label{lem:ubound_NC}
  Suppose that Assumptions \ref{assu:markov}-\ref{assu:RSW} hold.
  Let $\varepsilon,\delta$ with $0<10\varepsilon< \delta < 1$. Then there is a constant $C > 0$ such
  that for all $\eta \in(0, \varepsilon)$ we have
  \[
    \PP_{\eta}(\noClu(\varepsilon,\delta)^{c}) \leq C \frac{\varepsilon}{\delta^2}.
  \]
\end{lemma}

\begin{proof}[Proof of Lemma \ref{lem:ubound_NC}]
  A simple union bound combined with Lemma \ref{assu:arm_exp} provides the desired result.
\end{proof}

\section{Construction of the set of large clusters in the scaling limit}\label{sec:cont_clus_constr}

Now we are ready to construct the limiting object from Theorem \ref{thm:main_Hausdorff}. Before we do so, 
we note that Corollary \ref{cor:comp_prob_meas_space}, combined with Assumption \ref{assu:existence_full_confInvLimit}
and Lemma \ref{lem:convergenceQuadCrossings}, implies that there is a coupling of $\omega_\eta$'s for $\eta\geq0$,
denoted by $\PP$, such that
\begin{align*}
  \PP(\omega_\eta\rightarrow \omega_0 \text{ as } \eta\rightarrow 0) = 1,
\end{align*}
where $\omega_0$ has law $\PP_0$.

Fix some $\delta>0$. Let $\omega\in \mathscr{H_{1}}$ be a quad-crossing configuration. We define
\begin{align*}
  n_0(\omega) := \inf \left\{n \geq 0 \, | \, \omega\in \calE(3^{-n'},\delta) \text{ for all } n'\geq n\right\},
\end{align*}
where we use the convention that the infimum of the empty set is $\infty$ and the event $\calE(\varepsilon,\delta)$
is defined in \eqref{eq:def_E}. It is clear that $\calE(3^{-n},\delta) \in Borel(\mathscr{T}_{\hat\CC})$, hence the function $n_0$ is $Borel(\mathscr{T}_{\hat\CC})$ measurable.
Note that $\omega_\eta\in \calE(\eta/10,\delta)$ for $0<\eta<10\delta$. Hence $n_0(\omega_\eta) <\infty$ for all $0<\eta<10\delta$.
Furthermore, we write $g_n(\omega,\delta)$ for the number of good subgraphs in $G_{3^{-n}}(\omega)$.

Let $\eta>0$, $n \geq n_0(\omega_\eta)$, and $H^\eta$ be a good subgraph in  $G^\eta_{3^{-n}} = G_{3^{-n}}(\omega_\eta)$.
Proposition \ref{prop:good_sub} shows that for all $n'\geq n$, there is
a unique good subgraph $H'^\eta$ of $G^\eta_{3^{-n'}}$ such that $U(H^\eta)\supseteq U(H'^\eta)$.

Let $g_n^\eta = g_n(\omega_\eta,\delta)$.
For each $n\geq 0$, we fix an ordering of the graphs with vertex sets in $B_{3^{-n}}$.
For $j=1,2,\ldots,g_{n_0}^\eta$, let $H^\eta_{j,n_0} := H_{j,n_0(\omega_\eta)}(\omega_\eta)$ denote the $j$th good subgraph of $G^\eta_{3^{-n_0}}$.
Then for $n\geq n_0(\omega_\eta)$, let $H^\eta_{j,n}$ denote the unique good subgraph of $G_{3^{-n}}^\eta$ such that $U(H^\eta_{j,n_0})
\supseteq U(H^\eta_{j,n})$.

For $\eta\geq0$ and $j=1,2,\ldots,g_{n_0}^\eta$ we set 
\begin{align} \label{eq:def_contCl}
  \cl^\eta_j(\delta) := \bigcap_{n\geq n_0(\omega_\eta)} U(H^\eta_{j,n})
\end{align}
on the event $n_0(\omega_\eta)<\infty$, while on the event $n_0(\omega_\eta) = \infty$ we set $\cl_j^{\eta}(\delta) = \{-1/2,1/2\}$ for all $j\geq1$. (Note that we can replace $\{-1/2,1/2\}$ by any disconnected subset of $\Ball_1$.)
Since the sequence of compact sets $U(H^\eta_{j,n})$ is decreasing, the intersection in \eqref{eq:def_contCl} is non-empty on the event $n_0(\omega_\eta)<\infty$.
Proposition \ref{prop:good_sub} shows that for $\eta>0$, we get the collection of clusters $\clColl_1^{\eta}(\delta)$, that is,
\begin{align*}
  \clColl_1^{\eta}(\delta) = \{\cl^\eta_j(\delta) \, :\, 1\leq j\leq g_{n_0}^\eta\}.
\end{align*}

Before we state and prove the following precise version of Theorem \ref{thm:main_Hausdorff}, let us comment on the topology used there. We employ
a slightly different topology than the one in \eqref{eq:def_coll_dist}, defined as follows.

Let $\frC$ denote the set of non-empty closed subsets of $\Ball_1$ endowed with the Hausdorff distance $d_H$ as defined in \eqref{eq:def_Haus_top}. Let $l(\frC)$
denote the space of sequences in $\frC$. We endow it with the metric $d_l$ defined as
\begin{align} \label{def:dist_clVect}
  d_l(\underline{C},\underline{C}') := \sum_{j=1}^\infty d_H(C_j,C_j')2^{-j}
\end{align}
for $\underline{C} = (C_j)_{j\geq1},\underline{C}' = (C'_j)_{j\geq1}$.
Note that convergence in $d_l$ is equivalent with coordinate-wise convergence. Furthermore, $l^{\infty}(\frC)$ inherits the compactness from $\frC$.

For $\eta\geq0$, we extend the definition \eqref{eq:def_contCl} by setting $\calC^\eta_j(\delta) := \{-1/2,1/2\}$
for $j>g_{n_0}^\eta$. We write $\underline{\calC}_1^\eta(\omega_{\eta},\delta) := (\calC^\eta_j(\delta))_{j\geq1}$.

\newcommand{\allCl}{\underline{\calC}}
For a quad-crossing configuration $\omega$, $\allCl_{1}^{\eta} = \allCl_{1}^{\eta}(\omega)$ denotes the vector of all (macroscopic) clusters in $\omega$ defined as follows.
The first $g_{n_0}(\omega,3^{-1})$ entries of $\allCl_{1}^{\eta}(\omega)$ coincide with those of $\underline\calC_{1}^{\eta}(\omega,3^{-1})$. For $m\geq4$,
the next $g_{n_0}(\omega,m^{-1})-g_{n_0}(\omega,(m-1)^{-1})$
entries coincide with those elements in $\underline\calC_{1}^{\eta}(\omega,m^{-1})$
which are not listed earlier in $\allCl_{1}^{\eta}(\omega)$, with their relative order.
  
Now we are ready to state the following precise and slightly stronger version of Theorem \ref{thm:main_Hausdorff}.
\begin{thm} \label{thm:strong_Hausdorff}
  Suppose that Assumptions \ref{assu:markov}-\ref{assu:existence_full_confInvLimit} hold.
  Let $\delta>0$ and let $\PP$ be a coupling such that $\omega_\eta \rightarrow \omega_0$ a.s. as $\eta \rightarrow0$.
  Then $\underline{\calC}_1^\eta(\delta) \rightarrow \underline{\calC}_1^0(\delta)$ in probability in the metric $d_l$ as $\eta \rightarrow 0$.
  In particular, the pair $(\omega_\eta, \underline{\calC}_1^\eta(\delta))$ converges in distribution to $(\omega_0,\underline{\calC}_1^0(\delta))$ as $\eta\rightarrow 0$.
  Moreover, the same convergence result holds for $\underline{\calC}_1^\eta$. Furthermore, $\underline{\calC}_1^0(\delta)$ and $\allCl_1^0$ are measurable
  functions of $\omega_0$.
\end{thm}

\begin{rem}
  Note that the connected sets of $\Ball_1$ form a compact subspace of $\frC$. Hence $\{-1/2,1/2\}$ is
  separated from the clusters $\calC_j^\eta$ for $j=1,\ldots,g_{n_0}^\eta$. Thus the convergence of the vectors
  $\ul\calC_1^\eta(\delta)$ in the metric $d_l$ implies the convergence of $\clColl_1^\eta(\delta)$ in the topology
  \eqref{eq:def_coll_dist}. Namely, the bijection is given by the ordering of the entries in the corresponding vectors, 
  while the proof of Lemma \ref{lem:bound_cluster_count} implies that, in the sequence, there
  is no pair of clusters converging to the same closed set. The convergence in the metric \eqref{eq:def_coll_dist_hat_CC}
  follows from the equivalence of the metrics $d_H$ and $D_H$.
\end{rem}

Before we turn to the proof of Theorem \ref{thm:strong_Hausdorff}, we prove the following lemma.
\begin{lemma} \label{lem:n_0}
  Suppose that Assumptions \ref{assu:markov}-\ref{assu:existence_full_confInvLimit} hold.
  Let $\PP$ be a coupling such that $\omega_\eta\ra\omega_0$ $\PP$-a.s. as $\eta\ra0$. Then 
  \begin{align*}
    \PP(n_0(\omega_0) = \infty) = 0.
  \end{align*}
  Moreover, $n_0(\omega_\eta)\ra n_0(\omega_0)$ in probability under $\PP$ as $\eta\ra0$.
\end{lemma}
\begin{proof}[Proof of Lemma \ref{lem:n_0}]
  For each fixed $\varepsilon,\delta>0$ the event $\calE(\varepsilon,\delta)$ can be
  written as a finite union of intersections of some events appearing in Lemma
  \ref{lem:GPS}. Thus
  \begin{align*}
    \PP_0(\calE(\varepsilon,\delta)^{c}) &= \lim_{\eta\rightarrow0}\PP_\eta(\calE(\varepsilon,\delta)^{c})
    \leq C\varepsilon^\lambda
  \end{align*}
  with $C$ and $\lambda$ as in Proposition \ref{prop:ubound_calE}. Hence
  \begin{align*}
    \sum_{n=1}^\infty\PP_0(\calE(3^{-n},\delta)^{c})<\infty.
  \end{align*}
  Thus the Borel-Cantelli lemma shows that $\PP(n_0(\omega_0) = \infty) = 0$.

  Let $k\geq1$. Lemma \ref{lem:GPS} and Proposition \ref{prop:ubound_calE} imply that
  \begin{multline} \label{eq:pf_lem_n_0_1}
    \PP(|n_0(\omega_\eta)-n_0(\omega_0)| \geq 1)\\
    \begin{aligned}
      & \leq \PP(n_0(\omega_\eta) > k) + \PP(n_0(\omega_0)>k) \\
      & \quad+ \PP(|n_0(\omega_\eta)-n_0(\omega_0)| \geq 1, n_0(\omega_0)\vee n_0(\omega_\eta) \leq k) \\
      & \leq \sum_{l\geq k + 1}\left(\PP_\eta(\calE(3^{-l},\delta)^{c})+ \PP_0(\calE(3^{-l},\delta)^{c})\right) \\
      & \quad + \PP\big(\exists l\leq k \text{ s.t. } \ind_{\{\omega_\eta\in\calE(3^{-l}, \delta)\}}\neq \ind_{\{\omega_0\in\calE(3^{-l}, \delta)\}}\big) \\
      & \leq C \sum_{l\geq k + 1}3^{-\lambda l} + \sum_{l=1}^k\PP\big(\ind_{\{\omega_\eta\in\calE(3^{-l}, \delta)\}}\neq \ind_{\{\omega_0\in\calE(3^{-l}, \delta)\}}\big)
    \end{aligned}
  \end{multline}
  with some constant $C>0$. Taking $\eta\ra0$ in \eqref{eq:pf_lem_n_0_1} with a suitable constant $C'$ we get
  \begin{align*}
    \lim_{\eta\ra0}\PP(|n_0(\omega_\eta)-n_0(\omega_0)| \geq 1) & \leq C'3^{-\lambda k}
  \end{align*}
  for all $k>0$. This shows that $n_0(\omega_\eta)\ra n_0(\omega_0)$ in probability as $\eta\ra0$, and concludes the proof of Lemma \ref{lem:n_0}.
\end{proof}

\begin{proof}[Proof of Theorem \ref{thm:strong_Hausdorff}]
  Let $\delta>0$ and let $\PP$ be a coupling such that $\omega_\eta\ra\omega_0$ a.s. We will work under $\PP$ in what follows.
  Note that for each $n\in \NN$, the event $\calE(3^{-n}, \delta)$, the graph $G_{3^{-n}}(\omega)$ and the good subgraphs of $G_{3^{-n}}(\omega)$
  are functions of the outcomes of finitely many arm events appearing in Lemma \ref{lem:GPS}. Thus, as $\eta \to 0$, each of
  \begin{itemize}
    \item $\ind_{\left\{\omega_{\eta}\in\calE(3^{-n}, \delta)\right\}}$,
    \item $G_{3^{-n}}(\omega_{\eta})$, and
    \item the ordered set of good subgraphs of $G_{3^{-n}}(\omega_{\eta})$
  \end{itemize}
  converges in probability to the same quantity with $\omega_{\eta}$ replaced by $\omega_0$. This implies the following convergence
  statements in probability as $\eta \to 0$:
  \begin{enumerate}[1)]
    \item \label{enu:pf_strong_Hausdorff_1} by Lemma \ref{lem:n_0}, $n_0(\omega_{\eta})\rightarrow n_0(\omega_0)<\infty$,
    \item \label{enu:pf_strong_Hausdorff_2} $g_{n}^{\eta} \rightarrow g_{n}^{0}$ for all $n\geq1$, in particular, $g_{n_0(\omega_{\eta})}^{\eta} \rightarrow g_{n_0(\omega_0)}^{0}$,
    \item \label{enu:pf_strong_Hausdorff_3} $H^{\eta}_{j,n}\rightarrow H^{0}_{j,n}$ for $j = 1,2,\ldots,g_{n_0(\omega_0)}$ and $n\geq n_0(\omega_0)$.
  \end{enumerate}
  Let $n\geq n_0(\omega_{\eta})\vee n_0(\omega_0)$, then
  \begin{align} \label{eq:pf_strong_Hausdorff_1}
    d_H(\calC_j^{\eta},\calC_j^0) & \leq d_H(\calC_j^{\eta},U(H_{j,n}^{\eta})) + d_H(U(H_{j,n}^{\eta}),U(H_{j,n}^{0})) + d_H(U(H_{j,n}^{0}),\calC_j^0) \nonumber\\
    & \leq 3^{-n} + d_H(U(H_{j,n}^{\eta}),U(H_{j,n}^{0})) + 3^{-n} 
  \end{align}
  for $j = 1,2,\ldots,g^{\eta}_{n_0}\wedge g^{0}_{n_0}$. Thus taking the limit $\eta\ra0$ in \eqref{eq:pf_strong_Hausdorff_1},
  by \ref{enu:pf_strong_Hausdorff_1})-\ref{enu:pf_strong_Hausdorff_3}) above, we get
  \begin{align} \label{eq:pf_strong_Hausdorff_2}
    \lim_{\eta\rightarrow 0} \PP(d_H(\calC_j^{\eta},\calC_j^0) > 3\cdot3^{-n}, n\geq n_0(\omega_0)\vee n_0(\omega_\eta)) = 0 
  \end{align}
  for $j\geq1$. Then taking the limit $n\ra \infty$, Lemma \ref{lem:n_0} shows that $\calC_j^{\eta}\ra\calC_j^0$ in
probability in the Hausdorff metric as $\eta\ra0$ for all $j\geq1$. 
  Since convergence in $l^{\infty}(\frC)$ coincides with coordinate-wise convergence, we get
  that $\lim_{\eta\rightarrow 0}\underline{\calC}_{1}^{\eta}(\delta) = \underline{\calC}_1^0(\delta)$ in probability, as required.
  
  The proof of the claims of Theorem \ref{thm:strong_Hausdorff} for $\allCl_1^\eta$ is analogous. It follows from the convergence of $\underline{\calC}_1^\eta(\delta)$ with
  $\delta=3^{-m}$ for $m\geq1$.
  The measurability of $\underline{\calC}_1^0(\delta)$ and $\allCl_1^0$ with respect to $\omega_0$ follows easily from
  their definition involving arm events (see Remark \ref{rem:MeasurabilityArmEvents}). Thus the proof of Theorem \ref{thm:strong_Hausdorff}
  is complete.
\end{proof}

\section{Scaling limit in a bounded domain}\label{sec:clustersinSubset}
\begin{figure}
	\centering 
	\includegraphics[width = 0.80\textwidth]{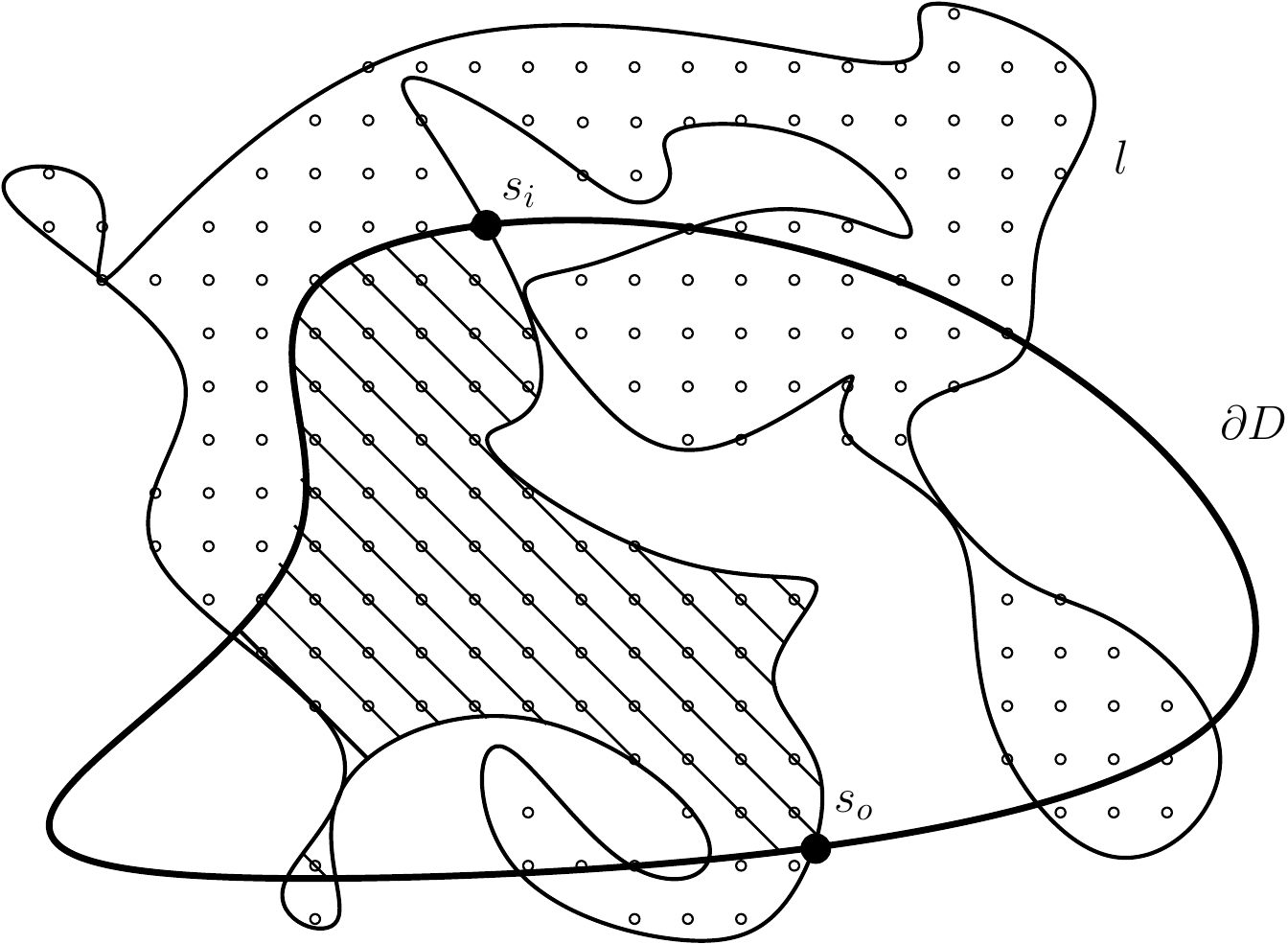}
	\caption{Illustration of a cluster in $D$. The small open circles denote the interior of the loop $l$.
	The shaded area intersected with the cluster of the loop is equal to $\mathcal{B}(\mathcal{E})$.}
	\label{fig:portionsOfClusters}
\end{figure}
In this section we will deduce the convergence of all clusters and ``pieces'' of clusters contained in a bounded domain $D$
from the convergence of clusters and loops completely contained in $\Ball_k \supset D$, for some $k$ sufficiently large.
We call $\mathscr{B}_{D}^{\eta}(\delta)$ the collection of all clusters or portions of clusters of diameter at least $\delta$
contained in $D^{\eta}$, where $D^{\eta}$ denotes an appropriate discretization of $D$. In the case of ${\ZZ}^2$, the boundary of $D^{\eta}$
is a circuit in the medial lattice that surrounds all the vertices of ${\ZZ}^2$ contained in $D$ and minimizes the distance to $\partial D$.
Analogously, in the case of the triangular lattice, $\TT$, the boundary of $D^{\eta}$ is a circuit in the dual (hexagonal) lattice that surrounds
all the vertices of $\TT$ contained in $D$ and minimizes the distance to $\partial D$.
More precisely, for every cluster $\cl \in \clColl^{\eta}(\delta)$ that intersect $D^{\eta}$, consider the set of all connected
components $\mathcal{B}$ of $\cl \cap D^{\eta}$ with diameter at least $\delta>0$.
For every $\eta, \delta >0$, we let $\mathscr{B}_{D}^{\eta}(\delta)$ denote the union of $\clColl_{D}^{\eta}(\delta)$
with the set of all such connected components $\mathcal{B}$.
(Note that clusters contained in $\Ball_k$ but not completely contained in $D^{\eta}$ are split
into different elements of $\mathscr{B}_{D}^{\eta}(\delta)$ (see Figure \ref{fig:portionsOfClusters}).
For the case of Bernoulli percolation, the collection $\mathscr{B}_{D}^{\eta}(\delta)$ is precisely the set of all clusters in $D^{\eta}$
with closed boundary condition.

As in Section \ref{sec:cont_clus_constr}, instead of the collection $\mathscr{B}_{D}^{\eta}(\delta)$, we consider
the sequence $\underline{\mathcal{B}}_{D}^{\eta}(\delta)$ of clusters with diameter at least $\delta$, with the metric $d_{l}$.
Now we are ready to state the theorem on the convergence of all portions of clusters in
$\sigma_\eta \cap D$ for a bounded domain $D$.
\begin{thm} \label{thm:strong_Hausdorff_in_domain}
  Suppose that Assumptions \ref{assu:markov}-\ref{assu:existence_full_confInvLimit} hold.
  Let $D$ be a simply connected bounded domain with piecewise smooth boundary.
  Let $\PP$ be a coupling where $(\omega_\eta,L_\eta) \rightarrow (\omega_0,L_0)$ a.s. as $\eta \rightarrow0$.
  Then, for any $\delta>0$, $\underline{\calB}_D^\eta(\delta) \rightarrow \underline{\calB}_D^0(\delta)$ in probability in
  the metric $d_l$ as $\eta \rightarrow 0$.
  In particular, the triple $(\omega_\eta,L_\eta, \underline{\calB}_D^\eta(\delta))$ converges in
  distribution to $(\omega_0,L_0,\underline{\calB}_D^0(\delta))$ as $\eta\rightarrow 0$.
  Moreover, the same convergence result holds for $\underline{\calB}_D^\eta$. Furthermore,
  $\underline{\calB}_D^0(\delta)$ and $\mathscr{B}_{D}^{0}$ are measurable
  functions of the pair $(\omega_0, L_0)$.
\end{thm}

\begin{proof}[Proof of Theorem \ref{thm:strong_Hausdorff_in_domain}]
 Let $(\omega_{\eta},L_{\eta})$ and $(\omega_{0},L_{0})$ be as in the statement of Theorem \ref{thm:strong_Hausdorff_in_domain}.
The probability that all the clusters that intersect $D$ are completely contained in $\Ball_{k}$
is at least one minus the probability of having a red arm from the boundary of $D$ to $\partial \Ball_{k}$.
The latter probability goes to zero as $k \to \infty$, hence there is a finite $k \in \NN$ such that
there is no red arm from $D$ to $\partial \Ball_{k-1}$ in $\omega_0$. We take the smallest such $k$. With this choice,
all clusters in $\clColl^{\eta}$ that intersect $D$ are contained in $\Ball_{k}$.

We first give an orientation to the loops contained in $\Ball_k$ in such a way that clockwise loops are the outer boundaries of red clusters
and counterclockwise loops are the outer boundaries of blue clusters. For each clockwise loop $\ell$ intersecting $\partial D$, we consider all excursions
$\mathcal{E}$ inside $D$ of diameter at least $\delta$. Each excursion $\mathcal{E}$ runs from a point $s_{in}$ on $\partial D$ to a point
$s_{out}$ on $\partial D$. We call the counterclockwise segment of $\partial D$ from $s_{in}$ to $s_{out}$ the base of $\mathcal{E}$.
We call $\overline{\mathcal{E}}$ the concatenation of $\mathcal{E}$ with its base. We define the interior ${\text I}(\overline{\mathcal{E}})$ of
$\overline{\mathcal{E}}$ to be the closure of the set of points with nonzero winding number for the curve $\overline{\mathcal{E}}$.

We call $\mathscr{E}_{\mathcal{E}}$ the collection of all clockwise excursions in $D$ of the same loop $\ell$ with base contained inside the base of $\mathcal{E}$.
If $\cl$ is the cluster whose outer boundary is the loop $\ell$, we define $\mathcal{B}(\mathcal{E})$ as follows:
\begin{equation} \nonumber
\mathcal{B}(\mathcal{E}) := \overline{ {\text I}(\overline{\mathcal{E}}) \setminus \{ \cup_{\mathcal{E'} \in \mathscr{E}_{\mathcal{E}}} {\text I}(\overline{\mathcal{E'}}) \}} \cap \cl ,
\end{equation}
where by $\cup_{\mathcal{E'} \in \mathscr{E}_{\mathcal{E}}} {\text I}(\overline{\mathcal{E'}})$ we mean
$\lim_{\xi \to 0} \cup_{\mathcal{E'} \in \mathscr{E}_{\mathcal{E}}, \diam{\mathcal{E'}}>\xi} {\text I}(\overline{\mathcal{E'}})$, and the limit exists
because it is the limit of an increasing sequence of closed sets.

For any $\delta>0$, $\mathscr{B}_{D}^{0}(\delta)$ is the collection of all sets $\mathcal{B}(\mathcal{E})$ defined above, for all clockwise excursions
$\mathcal{E}$ in $D$ of diameter at least $\delta$.

For any $\eta>0$, the collection $\mathscr{B}_{D}^{\eta}(\delta)$ contains all clusters completely contained in $D$ plus all the connected components
of the intersections of clusters in $\Ball_{k}$ with $D$. $\mathscr{B}_{D}^{\eta}(\delta)$ can be obtained with the following construction which
mimics the continuum construction given earlier. We first give an orientation to the loops contained in $\Ball_k$ in such a way
that loops that have red in their immediate interior are oriented clockwise and loops that have blue in their immediate interior are
oriented counterclockwise. For each clockwise loop $\ell^{\eta}$ intersecting $\partial D^{\eta}$, we consider all excursions $\mathcal{E}^{\eta}$ inside $D^{\eta}$ of
diameter at least $\delta$. Each excursion $\mathcal{E}^{\eta}$ runs from a point $s^{\eta}_{in}$ on $\partial D^{\eta}$ to a point
$s^{\eta}_{out}$ on $\partial D^{\eta}$. We call the counterclockwise segment of $\partial D^{\eta}$ from $s^{\eta}_{in}$ to $s^{\eta}_{out}$ the base of $\mathcal{E}^{\eta}$.
We call $\overline{\mathcal{E}^{\eta}}$ the concatenation of $\mathcal{E}^{\eta}$ with its base. We define the interior ${\text I}\left(\overline{\mathcal{E}^{\eta}}\right)$ of
$\overline{\mathcal{E}^{\eta}}$ to be the set of hexagons contained inside $\overline{\mathcal{E}^{\eta}}$.

We call $\mathscr{E}^{\eta}_{\mathcal{E}^{\eta}}$ the collection of all clockwise excursions in $D^{\eta}$ of the same loop $\ell^{\eta}$ with base contained inside the base of
$\mathcal{E}^{\eta}$. If $\cl^{\eta}$ is the cluster whose outer boundary is the loop $\ell^{\eta}$, we define $\mathcal{B}^{\eta}(\mathcal{E}^{\eta})$ as follows:
\begin{equation} \nonumber
\mathcal{B}^{\eta}(\mathcal{E}^{\eta}) := \overline{ {\text I}\left(\overline{\mathcal{E}^{\eta}}\right) \setminus \left\{ \cup_{(\mathcal{E}^{\eta})' \in \mathscr{E}^{\eta}_{\mathcal{E}^{\eta}}} {\text I}\left(\overline{(\mathcal{E}^{\eta})'}\right) \right\}} \cap \cl^{\eta} .
\end{equation}

We now note that the almost sure convergence
$(\omega_\eta,L_\eta) \rightarrow (\omega_0,L_0)$, combined with Lemma \ref{assu:arm_exp},
implies the same for the excursions in $D$. (Lemma \ref{assu:arm_exp} insures, via
standard arguments, that an excursion cannot come close to the boundary of $D$ without
touching it, so that large lattice and continuum
excursions will match with high probability for $\eta$ sufficiently small.
For more details on how to use Lemma \ref{assu:arm_exp}, the interested reader is
referred to Lemma 6.1 of \cite{Camia2006}.) Together with the convergence of
the clusters, this implies that $(\omega_\eta,L_\eta, \underline{\calB}_D^\eta(\delta))$ converges
in distribution to $(\omega_0,L_0,\underline{\calB}_D^0(\delta))$ as $\eta\rightarrow 0$,
the ordering is simply given by the ordering of the clusters completely contained in $D$ and
a clockwise ordering of the points $s_{in}$ ($s_{in}^{\eta}$).
The above result is valid for any $\delta>0$, so letting $\delta \to 0$ gives the second part of the theorem.
\end{proof}

\section{Limits of counting measures of clusters} \label{sec:cont_meas}
In this section we state and prove Theorem \ref{thm:strong_meas}, a precise and slightly
stronger version of Theorem \ref{thm:main_meas}.
We do this for the more general case of (portions of) clusters $\mathscr{B}_{D}^{\eta}(\delta)$ in a
domain with piecewise smooth boundary $D$.
The convergence of measures of the clusters which are completely contained in $\Ball_k$
follows immediately. For ease of notation we assume $D$ to be $\Ball_1$. 


Let $\frM$ denote the set of finite Borel measures on $\Ball_1$ endowed with the Prokhorov metric. Recall that $\frM$ is a separable metric space.

For $\eta \ge 0$, $n \in \NN$ and $S\subseteq \Ball_1$, we define
\begin{align} \label{eq:def_mu_S,ve}
  \mu_{S,n}^\eta := \sum_{z\in\ZZ[\imag]:\Ball_{3\cdot 3^{-n}/2}(3^{-n} z)\cap S\neq \emptyset}\mu^\eta_{1,A(3^{-n} z;3^{-n}/2,\delta/2-3^{-n})}.
\end{align}
This is the sum of counting measures $\mu_{1,A(z;3^{-n}/2,\delta/2-3^{-n})}^\eta$ such that $z \in 3^{-n}\ZZ[\imag]$ and
the inner box $\Ball_{3^{-n}/2}(z)$ or one of its neighbors has nonempty intersection with $S$.

Simple arguments show the following:
\newcommand{\allClMe}{\underline{\mu}}
\begin{obs} \label{obs:monotone_meas}
  Let $B$ be a Borel subset of $\CC$ and
  $S \subseteq \Ball_1$. Then, for fixed $\eta>0$,
  $\mu^\eta_{S,n}(B)\geq \mu^\eta_{S,n'}(B)$ for 
  $n'\geq n$ with probability $1$.
\end{obs}
It is easy to check that, for all fixed $\eta > 0$ and $\calB \in \mathscr{B}_{\Ball_1}^{\eta}(\delta)$, the following limit exists
\begin{align} \label{eq:def_mu_j}
  \lim_{n\rightarrow\infty} \mu^\eta_{\calB,n}
\end{align}
and is actually equal to $\mu_\calB^\eta$ as defined in \eqref{eq:def_mu_C}.
This motivates us to define, for any cluster $\calB \in \mathscr{B}_{\Ball_1}^{0}(\delta)$, $\mu_\calB^0$ by \eqref{eq:def_mu_j}
with $\eta = 0$ if the limit exists, and set $\mu_\calB^0 = 0$ when it does not.

Let $l(\frM)$ denote the set of infinite sequences in $\frM$ with bounded distance from the empty measure. Similarly to \eqref{def:dist_clVect}, we set 
\begin{align*}
  d_l(\underline{\nu},\underline{\phi}) := \sum_{j=1}^\infty \frac{d_P(\nu_j,\phi_j)}{1+d_P(\nu_j,\phi_j)}2^{-j}
\end{align*}
for $\underline{\nu},\underline{\phi}\in l(\frM)$. It is easy to check that $l(\frM)$ is separable,
but not compact. Let $h^{\eta}(\delta) := |\mathscr{B}_{\Ball_1}^{\eta}(\delta)|$, for $\eta \ge 0$.
It follows from Lemma \ref{lem:n_0}, together with the tightness of the number of excursions of diameter
at least $\delta$ in $\Ball_{1}$, that $h^{0}(\delta)$ is a.s. finite.
For $\eta\geq0$, we define $\underline{\mu}^{\eta} = (\mu_j^{\eta})_{j\geq1}$,
the vector of measures $\mu_j^{\eta} := \mu_{\calB_j}^{\eta}$ for $\calB_j \in \mathscr{B}_{\Ball_1}^{\eta}(\delta)$
and $j = 1,2,\ldots,h^{\eta}(\delta)$, and we set $\mu_j^\eta = 0$ for $j > h^{\eta}(\delta)$.
We define $\allClMe^\eta$ similarly to $\allCl^\eta$.

Now we are ready to state the main result from this section.
\begin{thm}\label{thm:strong_meas}
  Suppose that Assumptions \ref{assu:markov}-\ref{assu:existence_full_confInvLimit} hold.
  Let $D$ be a simply connected bounded domain with piecewise smooth boundary.
  Let $\PP$ be a coupling such that $(\omega_\eta,L_\eta) \rightarrow (\omega_0,L_0)$ a.s. as $\eta \rightarrow0$.
  Then $\underline{\mu}_D^\eta(\delta)\rightarrow \underline{\mu}_D^0(\delta)$ in
  probability as $\eta \rightarrow 0$, where $\underline{\mu}_D^0(\delta)$ is a measurable function of the pair
  $(\omega_0,L_0)$. In particular, the triple $(\omega_\eta,L_\eta, \underline{\mu}_D^\eta(\delta))$ converges in
  distribution to $(\omega_0,L_0,\underline{\mu}_D^0(\delta))$ as $\eta\rightarrow 0$.
  The same convergence result holds when $\underline{\mu}_D^\eta(\delta)$ is replaced by $\allClMe_D^\eta$.
\end{thm}

The same conclusions hold for the measures of the clusters in $\hat{\CC}$
which intersect a bounded domain $D$, that is, keeping the information of connections outside $D$.

\begin{rem}
    Lemma \ref{lem:big_then_thick} below shows that clusters whose diameter is at least $\delta>0$ have nonzero mass.
    Thus the convergence in Theorem \ref{thm:strong_meas} implies convergence in the metric analogous to \eqref{eq:def_coll_dist}
    based on the Prokhorov metric $d_P$, and so Theorem \ref{thm:main_meas} is proved.
\end{rem}

Let us first show that Theorem \ref{thm:main_Hausdorff_fullPlane} follows easily
from Theorems \ref{thm:strong_Hausdorff} and \ref{thm:strong_meas}.
\begin{proof}[Proof of Theorem \ref{thm:main_Hausdorff_fullPlane}]
The proof is analogous to the proof of Theorem 6 of \cite{Camia2006}, so we only give a sketch.
Let $D$ be any bounded subset of ${\CC}$ and $k_1>k_2$ be such that $D \subset \Ball_{k_2}$.
The measures $\PP_{k_1}$ and $\PP_{k_2}$ can be coupled in such a way that they coincide inside $D$,
in the sense that they induced the same marginal distribution on $(\clColl_D^0,\muColl_D^0)$.
This is because they are obtained from the scaling limit of the same full-plane lattice measure $\PP_{\eta}$.
The consistency relations needed to apply Kolmogorov's extension theorem are then satisfied, which insures
the existence of a limit $\PP$.
\end{proof}

The following lemma plays an important role in the proof of Theorem \ref{thm:strong_meas}. Let $||\nu||_{TV}$ denote the total variation of a signed measure
$\nu$.
\begin{lemma} \label{lem:bound_err}
  Suppose that Assumptions \ref{assu:markov}-\ref{assu:RSW} hold. Let $\delta > 0$. Then there are positive
  constants $C = C(\delta),\varphi$ such that, for $n \in \NN$ and $\eta > 0$ with $0< 10\eta < 3^{-n} < \delta/10$,
  \begin{align*}
    \PP_{\eta}(\exists \calB\in\mathscr{B}_{\Ball_1}^{\eta}(\delta), S\subseteq \Ball_1 \text{ s.t.}\,\, d_{H}(\calB,S) <\varepsilon/2, ||\mu_\calB^\eta - \mu_{S,n}^{\eta} ||_{TV} \geq \varepsilon^{\varphi})\leq C\cdot \varepsilon^{\varphi}
  \end{align*}
  where $\varepsilon = 3^{-n}$.
\end{lemma}

\begin{proof}[Proof of Theorem \ref{thm:strong_meas} given Lemma \ref{lem:bound_err}]
Let $\PP$ be as in Theorem \ref{thm:strong_meas}, $\delta>0$.
It follows from Theorem \ref{thm:strong_Hausdorff_in_domain} that the clusters
in $\mathscr{B}_{\Ball_1}^{\eta}(\delta)$ converge in probability as $\eta \to 0$.

Moreover, Theorem \ref{assu:conv_measures} shows that each of the measures
\begin{align*}
  \mu^{\eta}_{1,A(3^{-n} z;3^{-n}/2,\delta/2-3^{-n})} \quad \text{for } n\geq1 \text{ and } z\in\ZZ[\imag] \text{ with } 3^{-n}z\in\Ball_1 .
\end{align*}
converges in probability in the Prokhorov metric, as $\eta\ra0$, to the analogous measure where $\eta$ is replaced by $0$.

This
implies that, for all fixed $n$ and $S \subset \Ball_1$, $\mu_{S,n}^\eta \ra \mu_{S,n}^0$ weakly in probability as $\eta\ra 0$.
The monotonicity of the measures $\mu_{S,n}^\eta$ in $n$ for a fixed subset $S$ and fixed $\eta$ (Observation \ref{obs:monotone_meas})
carries through to the limit as $\eta\ra0$, thus the weak limit $\mu_S^0 = \lim_{n\ra\infty}\mu_{S,n}^0$ exists almost surely.
Furthermore, since each of the measures $\mu_{S,n}^0$ is a function of $(\omega_0,L_0)$ and is a.s. finite, we conclude that $\mu_S^0$
is a.s. finite and is a function of $(\omega_0,L_0)$.

Let $\calB$ be the $j$-th element of $\underline{\calB}_{\Ball_1}^{0}(\delta)$ and
 let $\calB_{j}^{\eta}$ be the $j$-th element of $\underline{\calB}_{\Ball_1}^{\eta}(\delta)$,
where $\underline{\calB}_{\Ball_1}^{0}(\delta)$ and $\underline{\calB}_{\Ball_1}^{0}$ are the sequences
of clusters that appear in Theorem \ref{thm:strong_Hausdorff_in_domain}.
Fix $\kappa >0$. Lemma \ref{lem:bound_err} implies that, for some constants $\varphi,C=C(\delta)$, for
$\kappa > \varepsilon^{\varphi}$, $\eta < \varepsilon/10$ and $3^{-n} = \varepsilon$, we have 
\begin{multline} \label{eq:pf_str_meas}
  \PP(d_P(\mu_{\calB}^0,\mu_{\calB_{j}^{\eta}}^{\eta}) > 3\kappa) \\
  \begin{aligned}
    \leq & \,\, \PP(d_P(\mu_\calB^0,\mu_{\calB,n}^0) > \kappa) + \PP(d_P(\mu_{\calB,n}^{0},\mu_{\calB,n}^{\eta}) > \kappa)\\
         & \quad + \PP(||\mu_{\calB,n}^{\eta} - \mu_{\calB_j^{\eta}}^{\eta} ||_{TV} > \kappa, d_{H}(\calB, \calB_{j}^{\eta}) < \varepsilon/2) + \PP(d_{H}(\calB, \calB_{j}^{\eta}) \ge \varepsilon/2) \\
    \leq & \,\, \PP(d_P(\mu_\calB^0,\mu_{\calB,n}^0) > \kappa) + \PP(d_P(\mu_{\calB,n}^{0},\mu_{\calB,n}^{\eta}) > \kappa) \\
	 & \quad + C\kappa + \PP(d_{H}(\calB, \calB_{j}^{\eta}) \ge \varepsilon/2)\\
  \end{aligned}  
\end{multline}
where
$d_P$ denotes the Prokhorov distance of Borel measures.

Now we take the limit first as $\eta\ra0$ then as $n\ra\infty$ in \eqref{eq:pf_str_meas}. From the arguments above and Theorem \ref{thm:strong_Hausdorff_in_domain} we deduce that
\begin{align*}
  \lim_{\eta\ra0}\PP(d_P(\mu_\calB^0,\mu_{\calB_{j}^{\eta}}^{\eta}) > 3\kappa) \leq C\kappa
\end{align*}
for all $\kappa>0$. Thus the measures $\mu_{\calB_{j}^{\eta}}^{\eta}$ tend to $\mu_\calB^0$ weakly in probability as $\eta\ra0$.

Recall that the convergence in $l^\infty(\frM)$ is equivalent to coordinate-wise convergence. Thus $\underline{\mu}^{\eta}(\delta)\ra\underline{\mu}^0(\delta)$ in\
probability
as $\eta\ra0$. We have already proved in the lines above that $\underline{\mu}^0(\delta)$ is a measurable function of $(\omega_0,L_0)$,
thus we deduced the results in Theorem \ref{thm:strong_meas} for $\underline{\mu}^\eta(\delta)$.

The results for $\allClMe^\eta$ follow from the lines above by arguments similar to those at the end of the proof of Theorem \ref{thm:strong_Hausdorff}.
This concludes the proof of Theorem \ref{thm:strong_meas}.
\end{proof}

We finish this section by proving Lemma \ref{lem:bound_err} above. Its proof relies on Lemma \ref{lem:BCKS}.

\begin{proof}[Proof of Lemma \ref{lem:bound_err}]
  Let $\eta,n,\delta$ as in Lemma \ref{lem:bound_err}. To simplify the notation, we set $\varepsilon := 3^{-n}$, 
  $\delta':=\delta/2-3\varepsilon$ and $\beta :=\frac{\lambda}{2(\lambda+\lambda_{1})}$, with  $\lambda_{1}$
  as in Lemma \ref{assu:1arm_exp_ubound} and $\lambda$ as in Lemma \ref{assu:1arm_exp_less_4arm}.

We define the following collection of `pivotal' boxes:
  \begin{align*}
    \Piv^\eta(\varepsilon,\varepsilon^{\beta}) :=  \{\Ball_{\varepsilon/2}(\varepsilon z)\,|\, z\in\ZZ[\imag]\cap\Ball_{\varepsilon^{-1}+1}; \omega_\eta\in\calA_{(1010),\emptyset}(\varepsilon z; 3\varepsilon/2,\varepsilon^{\beta})\}.
  \end{align*}
Furthermore, we let $\nu_{\varepsilon^\beta}^\eta$ denote the normalized counting measure of the vertices close to the boundary of $\Ball_1$ which have an open arm to distance $5\varepsilon^{\beta}$:
  \begin{align}\label{eq:def:nu_eta_eps}
   \nu^\eta_{\varepsilon^{\beta}} := \frac{\eta^2}{\pi_{1}^{\eta}(\eta,1)} \sum_{v\in A(0;1-\varepsilon^{\beta},1) \cap \eta V} \delta_{v}\ind\{v\xlra{(1)} \partial \Ball_{5\varepsilon^{\beta}}(v)\}.
  \end{align}
Roughly speaking, $\nu_{\varepsilon^\beta}^\eta$ is introduced to account for boxes near $\partial\Lambda_1$ where two large pieces
of a cluster come close to each other. Such boxes are not necessarily `pivotal' since the two large pieces may connect just outside $\Lambda_1$,
in which case the boxes are not counted in $\Piv^\eta$.
  
  Take $\calB\in\mathscr{B}_{\Ball_1}^{\eta}(\delta)$ and $S\subseteq \Ball_1$ such
  that $d_{H}(S,\calB)<\varepsilon/2$.
  Note that $d_{H}(S,\calB)<\varepsilon/2$ implies that the counting measure $\mu^\eta_{S,n}$
  is larger than or equal to the counting measure $\mu^\eta_{\calB}$.
  As a consequence it is easy to check that, for these $\calB$ and $S$, we have
  \begin{multline} \label{eq:pf_lem_bound_err_1}
    ||\mu^\eta_{S,n} - \mu^\eta_{\calB}||_{TV} \\
    \begin{aligned}
    & \leq ||\nu^{\eta}_{\varepsilon^{\beta}}||_{TV} + \sum_{z\in \ZZ[\imag] \,:\, \Ball_{\varepsilon/2}(\varepsilon z)\in \Piv^\eta(\varepsilon,\varepsilon^{\beta})} ||\mu^\eta_{1,A(\varepsilon z;3\varepsilon/2,\delta')}||_{TV} \\
    & \leq  ||\nu^{\eta}_{\varepsilon^{\beta}}||_{TV} + |\Piv^\eta(\varepsilon,\varepsilon^{\beta})|
      \sup_{z\in\ZZ[\imag]\cap\Ball_{\varepsilon^{-1}+1}} ||\mu^\eta_{1,A(\varepsilon z;3\varepsilon/2,3\varepsilon)}||_{TV}.
    \end{aligned}
  \end{multline}
  Letting $a^\eta_\varepsilon := \varepsilon^{-(2+\varphi)}\pi^\eta_4(3\varepsilon/2,\varepsilon^{\beta})$,
  from \eqref{eq:pf_lem_bound_err_1} we deduce that
  \begin{multline} \label{eq:pf_lem_bound_err_2}
    \PP_\eta(\exists \calB\in\mathscr{B}_{\Ball_1}^{\eta}(\delta), S \subseteq \Ball_1 \text{ s.t. }
    d_{H}(S,\calB)<\varepsilon/2, ||\mu_\calB^\eta - \mu_{S,n}^\eta ||_{TV} \geq \varepsilon^{\varphi}) \\
      \begin{aligned}
      &\leq\PP_{\eta}(||\nu^{\eta}_{\varepsilon^{\beta}}||_{TV} \ge \frac{1}{2}\varepsilon^{\varphi}) + \PP_\eta(|\Piv^\eta(\varepsilon,\varepsilon^{\beta})| \geq a^\eta_\varepsilon)\\
      & + \quad \PP_\eta(\sup_{z\in \Ball_{\varepsilon^{-1}+1}\cap \ZZ[\imag]} ||\mu^\eta_{1,A(\varepsilon z;3\varepsilon/2,3\varepsilon)}||_{TV} > \varepsilon^{\varphi}/2a_{\varepsilon}^{\eta}),
      \end{aligned}
  \end{multline}
for some $\varphi$ to be fixed later.
  By the Markov inequality, we have
  \begin{align} \label{eq:pf_lem_bound_err_3}
    \PP_\eta(|\Piv^\eta(\varepsilon,\varepsilon^{\beta})| \geq a^\eta_\varepsilon) \leq C_1 \varepsilon^{\varphi}
  \end{align}
  for some positive constant $C_1 = C_1(\delta)$ for all $\varphi>0$.

  Now we bound the third term in \eqref{eq:pf_lem_bound_err_2}. With some positive constants $C_2,C_3,C_4$ depending on $\delta$,
  and recalling Definition \ref{def:V_a}, we have that
  \begin{multline}\label{eq:pf_lem_bound_err_6} 
    \PP_\eta(\sup_{z\in \Ball_{\varepsilon^{-1}+1}\cap \ZZ[\imag]} ||\mu^\eta_{1,A(\varepsilon z;3\varepsilon/2,3\varepsilon)}||_{TV} > \varepsilon^{\varphi}/ 2a^\eta_\varepsilon) \\
    \begin{aligned}
      &\leq C_2 \varepsilon^{-2} \PP_\eta(||\mu^\eta_{1,A(3\varepsilon/2,3\varepsilon)}||_{TV} > \varepsilon^{\varphi} / 2a^\eta_\varepsilon) \\
      &= C_2 \varepsilon^{-2} \PP_\eta(|\calV^\eta_{3\varepsilon}|\geq \varepsilon^{\varphi} \eta^{-2}\pi_1^\eta(\eta,1) / 2a^\eta_\varepsilon)\\
      &\leq C_2 \varepsilon^{-2} \exp\left( -C_3 \varepsilon^{2\varphi}\frac{\pi_1^\eta(\eta,1)}{\pi_1^\eta(\eta,3\varepsilon) \pi_4^\eta(3\varepsilon/2,\varepsilon^{\beta})}\right)\\
      &\leq C_2 \varepsilon^{-2} \exp\left( -C_4 \varepsilon^{2\varphi}\frac{\pi_1^\eta(3\varepsilon,\varepsilon^{\beta})}{\pi_4^\eta(3\varepsilon/2,\varepsilon^{\beta})}\pi_1^\eta(\varepsilon^{\beta},1)\right),
    \end{aligned}
  \end{multline}
  where, in the second inequality, we used Lemma \ref{lem:BCKS} and, in the last line, we
  used Lemma \ref{assu:q_mult} twice. Lemmas \ref{assu:1arm_exp_ubound} and
  \ref{assu:1arm_exp_less_4arm}, \eqref{eq:pf_lem_bound_err_6} and the choice of $\beta$ give that
  \begin{align} \label{eq:pf_lem_bound_err_4}
    \PP_\eta(\sup_{z\in \Ball_{\varepsilon^{-1}+1}\cap \ZZ[\imag]} ||\mu^\eta_{1,A(\varepsilon z;3\varepsilon/2,3\varepsilon)}||_{TV} > \varepsilon^{\varphi}/ 2a^\eta_\varepsilon)
    & \leq C_2 \varepsilon^{-2}\exp(-C_5\varepsilon^{2\varphi+\lambda(\beta-1)+\lambda_1\beta})\nonumber\\
    & = C_2 \varepsilon^{-2}\exp(-C_5\varepsilon^{2\varphi-\lambda/2})
  \end{align}
  with $C_5 > 0$. Computations similar to those above give the following upper bound for the second term in \eqref{eq:pf_lem_bound_err_2}:

  \begin{align}
     \PP_{\eta}(||\nu^{\eta}_{\varepsilon^{\beta}}||_{TV} \geq \frac{1}{2}\varepsilon^{\varphi})
     & \leq C_{6} \varepsilon^{-\beta} \exp{\left(-C_{7}  \varepsilon^{\varphi-\beta} \frac{\pi_{1}^{\eta}(\eta,1)}{\pi_{1}^{\eta}(\eta,\varepsilon^{\beta})}\right)} \nonumber \\
     & \leq C_6 \varepsilon^{-\beta} \exp{\left(-C_{8} \varepsilon^{\varphi-\beta+\beta\lambda_1}\right)} \label{eq:pf_lem_bound_err_5}
  \end{align}
  for suitable constants $C_6,C_7,C_8$.
  We set $\varphi = \frac{\lambda \wedge (\beta(1-\lambda_{1}))}{4} >0$. A combination of \eqref{eq:pf_lem_bound_err_2}, \eqref{eq:pf_lem_bound_err_3},
  \eqref{eq:pf_lem_bound_err_4} and \eqref{eq:pf_lem_bound_err_5} finishes the proof of Lemma \ref{lem:bound_err}.
\end{proof}

\section{Properties of the continuum clusters and their normalized counting measures} \label{sec:props}
We start with the connections between the clusters and their counting measures. The first 
result of the section shows, roughly speaking, that the scaling limit of the clusters as 
closed sets contains the same information as their normalized counting measures.  Then we 
show conformal invariance of the clusters and conformal covariance of their normalized counting 
measures.
\subsection{Basic properties} \label{ssec:basic_props}
Recall the notation $\clColl^\eta(\delta)$ from \eqref{eq:def_full_clColl}. We set $\clColl^0 = \bigcup_{n=1}^\infty \clColl^0(3^{-n})$.
For $\cl\in\clColl^0$ and $0<\psi<1/2$ we write 
\begin{align} \label{eq:mu_recov}
 \tilde{\mu}_{\calC,\psi}^0
  & := \frac{4\psi^2}{\pi^0_1(2\psi,1)}\sum_{z\in\ZZ[\imag]:\Ball_{\psi/2}(\psi z)\cap\calC\neq\emptyset} \delta_{\psi z}.
\end{align}

\begin{thm} \label{thm:clus-meas}
  Suppose that Assumptions \ref{assu:markov}-\ref{assu:existence_full_confInvLimit} hold.
  Then $\supp(\mu_{\calC}^0) = \calC$ for all $\calC\in\clColl^0$. Moreover,
  \begin{align} \label{eq:thm_clus-meas}
    \tilde\mu_{\calC,\psi}^0 \ra \mu_{\cl}^0 \text{ weakly in probability as } \psi\ra0
  \end{align}
  for all $\calC\in\clColl^0$.
\end{thm}

The proof of the theorem above relies on the following two lemmas.
\begin{lemma}\label{lem:big_then_thick}
  Suppose that Assumptions \ref{assu:markov} - \ref{assu:RSW} hold.
  Let $k,\delta > 0$. Then for all $\varphi > 0$ there is $x_\varphi = x_\varphi(k,\delta) > 0$ such that
  \begin{align} 
	\PP_\eta(\exists \calC\in\mathscr{B}_k^\eta(\delta) \text{ with } ||\mu_\calC^\eta||_{TV} < x_\varphi) < \varphi
  \end{align}
  for all $\eta\in(0,\delta)$.
\end{lemma}
\begin{proof}[Proof of Lemma \ref{lem:big_then_thick}]
  For critical percolation the proof of Lemma \ref{lem:big_then_thick} follows from the 
  proof of \cite[Theorem 1.2]{BC13}: (3.18) of \cite{BC13} with $x = 0$ can be shown in the same manner as for $x > 0$.
  Alternatively, Lemma \ref{lem:big_then_thick} can be deduced from a combination of
  \cite[Theorem 3.1 (i), Theorem 3.3 (i) and Lemma 4.4]{BCKS01}, using tightness of the number of clusters of
  diameter at least $\delta$.
  
  It is easy to verify that actually all these arguments just need Assumptions \ref{assu:markov} - \ref{assu:RSW}.
\end{proof}

The second is essentially \cite[Proposition 4.13]{Garban2010} see also \cite[Eqn. (4.39)]
{Garban2010}. Let $A$ be the annulus $A = A(a,b )$ with $0 < a < b$ and $\cl\in \clColl^0$. 
For $\eta\geq0$ and $0<\psi<1/2$ we set
\begin{align*}
 \tilde{\mu}_{A,\psi}^\eta & := \frac{4\psi^2}{\pi^\eta_1(2\psi,1)}\sum_{z\in\ZZ[\imag]\cap \Ball_{\psi^{-1}a}} \ind\{\Ball_{\psi/2}(\psi z)\xlra{1} \partial \Ball_b\} \delta_{\psi z}.
\end{align*}

\begin{lemma}[Proposition 4.13 of \cite{Garban2010}] \label{lem:GPS4.13}
  Suppose that Assumptions \ref{assu:markov}-\ref{assu:existence_full_confInvLimit} hold.
  Let $f:\CC\ra\RR$ be a continuous function with compact support, and let $A = A(a,b)$ be an annulus with $0 < a < b$. Then
  \begin{align} \label{eq:mu_to_mu_L2}
    \tilde \mu_{A,\psi}^0(f) \ra \mu^0_{A}(f) \text{ in } L^2 \text{ as } \psi\ra0.
  \end{align}
\end{lemma}

\begin{rem}
  For the proof of Theorem \ref{thm:clus-meas}, convergence in probability is enough in \eqref{eq:mu_to_mu_L2}.
\end{rem}

\begin{proof}[Proof of Thm.\ref{thm:clus-meas}] Since $\clColl^0 = \bigcup_{n=1}^\infty \clColl^0(3^{-n})$ and
$\clColl^0(3^{-n}) = \bigcup_{k \in \NN} \clColl_{k}^0(3^{-n})$, to prove the first part of the theorem, it suffices to show
that $\supp(\mu_{\calC}^0)={\calC}$ with probability $1$ for all $\calC\in\clColl_k^0(\delta)$ for any fixed $\delta > 0$ 
and $k\in \NN$. We will work under a coupling $\PP$ such that $\om_\eta\ra \om_0$ a.s.

Equations \eqref{eq:def_mu_S,ve} and \eqref{eq:def_mu_j} show that, for all $\calC\in\clColl^0(\delta)$,
$\supp(\mu_{\calC}^0)$ is contained in the $(3^{-n})$-neighborhood of $\calC$ for every $n$, with probability $1$.
Hence, $\supp(\mu_{\calC}^0) \subseteq \calC$ for all $\calC\in\clColl^0(\delta)$ with probability $1$.

We turn to the proof of $\supp(\mu_{\calC}^0) \supseteq \calC$. Take $\varphi>0$ and $x_\varphi$ as in Lemma \ref{lem:big_then_thick}.
By covering $\Ball_k$ with at most $4(k/\ve)^{2}$ squares with side length $\ve$, we get
\begin{multline} \label{eq:pf_thm_clus-meas_1}
  \PP_\eta(\exists z\in \ZZ[\imag],\exists \calC\in\clColl^\eta(\delta) \text{ s.t. } \Ball_{\ve/2}(\ve z)\cap \calC \neq \emptyset \text{ and } \mu_\calC^\eta(\Ball_{\ve}(\ve z)) < x_\varphi)\\
  \begin{aligned}
    & \leq 4(k/\ve)^{2} \PP_\eta (\exists \calB\in\mathscr{B}_{\Ball_\ve}^\eta(\ve/2) \text{ with } ||\mu_\calB^\eta||_{TV} < x_\varphi)\\
    & \leq 4(k/\ve)^{2}\varphi.
  \end{aligned}
\end{multline}

By Theorem \ref{thm:strong_meas} we have that $\ul\mu^\eta(\delta) \to \ul\mu^0(\delta)$ in probability in the metric $d_l$
for all $\delta>0$ as $\eta\ra 0$. This, combined with the tightness of $|\clColl_k^0(\delta)|$, \eqref{eq:pf_thm_clus-meas_1}
and the Portmanteau theorem, gives that 
\begin{multline} \label{eq:pf_thm_clus-meas_2}
  \PP_0(\exists z\in \ZZ[\imag],\exists \calC\in\clColl_k^0(\delta) \text{ s.t. } \Ball_{\ve/2}(\ve z)\cap \calC \neq \emptyset \text{ and } \mu_\calC^0(\Ball_{\ve}(\ve z)) < x_\varphi)\\
  \leq 4(k/\ve)^{2}\varphi
\end{multline}
for all $\ve \in(0,\delta/10)$. We take the limit $\varphi\ra0$ in \eqref{eq:pf_thm_clus-meas_2} and get
\begin{align} \label{eq:pf_thm_clus-meas_3}
  \PP_0(\exists z\in \ZZ[\imag],\exists \calC\in\clColl_k^0(\delta) \text{ s.t. } \Ball_{\ve/2}(\ve z)\cap \calC \neq \emptyset \text{ and } \mu_\calC^0(\Ball_{\ve}(\ve z)) = 0) = 0,
\end{align}
which shows that $\supp(\mu_\calC^0)+\Ball_\varepsilon \supseteq\calC$ for all $\calC\in\clColl_k^0(\delta)$ with probability $1$ for each fixed $\ve > 0$. Thus
$\supp(\mu^0_\calC) \supseteq \calC$ for all $\calC\in\clColl^0$ with probability $1$, and finishes the proof of the first statement of Theorem \ref{thm:clus-meas}.

\medskip

Since the proof of \eqref{eq:thm_clus-meas} is analogous to that of Lemma \ref{lem:bound_err}, we only give a sketch. Let $\calC\in \clColl^0(\delta)$ with $\delta>0$,
and let $f:\CC\ra\RR$ be a continuous function with compact support. Recall the definition \eqref{eq:def_mu_S,ve} of $\mu_{\calC,n}^0$. We set  
\begin{align*}
 \bar{\mu}_{\calC,n,\psi}^0
 & := \sum_{z\in\ZZ[\imag]:\Ball_{3\ve/2}(\ve z)\cap\calC\neq\emptyset} \tilde\mu^0_{A(\ve z,\ve/2,\delta/2-\ve),\psi},
 \; \text{ with } \ve = 3^{-n}.
\end{align*}
Note that when we replace $\mu^0_{A(\ve z,\ve/2,\delta/2-\ve)}$ by $\tilde\mu^0_{A(\ve z,\ve/2,\delta/2-\ve),\psi}$
in the definition of $\mu_{\calC,n}^0$, we arrive at the measure $\bar{\mu}_{\calC,n,\psi}^0$. Thus, for any fixed $\ve>0$,
Lemma \ref{lem:GPS4.13} shows that $\bar{\mu}_{\calC,n,\psi}^0(f)$ and $\mu_{\calC,n}^0(f)$ are close to each other in
$L^2$ when $\psi$ is small. In particular, $\bar\mu_{\calC,n,\psi}^0 \ra \mu_{\calC,n}^0$ weakly in probability as $\psi\ra0$.

Arguments similar to those in the proof of Lemma \ref{lem:bound_err} give that
$\tilde{\mu}_{\calC,\psi}^0$ and $\bar{\mu}_{\calC,n,\psi}^0$ are close to each other in total variation distance (hence in Prokhorov distance as well) with high probability
when $\psi$ and $\ve=3^{-n}$ are both small.

By the proof of Theorem \ref{thm:strong_meas}, $\mu^0_{\calC,n}$ is close to  $\mu^0_\calC$ in Prokhorov distance with high probability
when $n$ is large. Thus
\begin{align*}
  \tilde{\mu}_{\cl,\psi}^0 \approx \bar{\mu}_{\cl,n,\psi}^0 \xra[\psi\ra0] \mu_{\cl,n}^0\xra[n\ra\infty]\mu_\cl^0,
\end{align*}
where the limits are in Prokhorov metric in probability, and $\tilde{\mu}_{\cl,\psi}^0 \approx \bar{\mu}_{\cl,n,\psi}^0$ means
that the Prokhorov distance between these measures is small with high probability when $\ve=3^{-n}$ and $\psi$ are both small.
Thus \eqref{eq:thm_clus-meas} follows, and Theorem \ref{thm:clus-meas} is proved.
\end{proof}

\subsection{Conformal invariance and covariance} \label{ssec:cinvar}
In this section we prove Theorem \ref{thm:invar} and the stronger conformal covariance
of Bernoulli percolation clusters as stated in Theorem \ref{thm:cinvar}.

Let us first restrict ourselves to critical site percolation on the triangular lattice.
At the end of this section we will show how to obtain the weaker invariance of Theorem \ref{thm:invar}
from our general assumptions. 

Recall Definition \ref{def:restr} of the restriction of a configuration to a bounded domain $D$.
\begin{thm}\label{thm:cinvar_perco}
 For $\eta \ge 0$, let $\PP_\eta$ denote the measure for critical site percolation on the triangular lattice.
 Let $D\subseteq \CC$ be a domain and $f: D \to \CC$ be a conformal map. The laws
 of $(f(\omega_{0,D}),f(L_{0,D}))$ and $(\omega_{0,f(D)},L_{0,f(D)})$ coincide.
\end{thm}
The conformal invariance of the continuum loop process was proved in \cite[Theorem 3, item 4]{Camia2006}.
The conformal invariance of the quad crossings follows immediately because of the measurability
with respect to the loop process \cite{Garban2010,Smirnov2011}.

The construction of the continuum clusters and their measures was obtained in Sections
\ref{sec:approx} - \ref{sec:cont_meas} by approximating the cluster by boxes $\Lambda_{\varepsilon/2}(z)$.
In order to prove conformal invariance / covariance we would like to approximate the clusters
by conformally transformed boxes $f(\Lambda_{\varepsilon/2}(z))$.
More precisely, let $\phi>0$ and $f:\Ball_{1+\phi}\ra \hat{\CC}$ be a conformal map.
We set $D = f(\Ball_1)$ and $D'=f(\Ball_{1+\phi})$.
Let $d_f$ denote the push-forward of the $L^\infty$ metric on $\Ball_{1+\phi}$. That is,
\begin{align*}
  d_f(x,y) := ||f^{-1}(x) - f^{-1}(y)||_\infty
\end{align*}
for $x,y\in D'$.
Note that $f$ is defined in an open neighborhood of $\Ball_1$ because when we approximate the cluster measures
using one arm measures, we need to consider annuli whose inner square is contained in $\Ball_1$ but which are
not completely contained in $\Ball_1$.

Clearly, $(\Ball_{1+\phi},d_\infty)$ and $(D',d_f)$ are isomorphic as metric spaces. Thus all the
geometric constructions in Section \ref{sec:approx} can be repeated for the clusters in $D$ just by applying the
map $f$. We denote these analogues of the objects by an additional `$f$' subscript. Thus all the statements,
apart from those in Section \ref{ssec:error_bounds}, remain valid if we keep the constants such as $\ve,\delta$
unchanged, but add an additional subscript $f$ in the objects appearing in the claims.
Moreover, the bounds in Section \ref{ssec:error_bounds} remain valid asymptotically, as $\eta \to 0$, if we use
the transformed boxes $f(\Lambda_{\varepsilon/2}(z))$ to define the relevant events because of the conformal
invariance of the scaling limit.

Next note that there is a positive constant $K = K(f)$ such that $|f'(u)|\in[1/K,K]$ for $u\in \Ball_{1+\phi/2}$.
Thus $d_f$ and the $L^\infty$-metric are equivalent on $D$. As above, we add a
subscript `$f$' for the metrics built from $d_f$. Thus $d_{H,f}$ and $d_{P,f}$ are
equivalent to $d_H$ and $d_P$ respectively, where $d_{H,f}$ and $d_{P,f}$ are built on $d_f$.

We can obtain the clusters in $D$ in two ways: via the square boxes $\Lambda_{\varepsilon/2}(z)$, that is,
using the metric $L^\infty$ in $D$, or via the transformed boxes $f(\Lambda_{\varepsilon/2}(z))$, that is,
using the metric $d_f$. The equivalence of the metrics implies that these two approximations provide the
same continuum clusters in the scaling limit.

Now notice that the scaling limit in $D$ in terms of quad crossings is distributed like the image under $f$
of the scaling limit in $\Ball_1$, because of the conformal invariance of quad crossing configurations.
This implies that the construction in $D$, using the transformed boxes $f(\Lambda_{\varepsilon/2}(z))$, gives
clusters that have the same distribution as the images of the continuum clusters in $\Ball_1$. This proves the following theorem.

\begin{thm}\label{thm:cinvar_prec}
 For $\eta \ge 0$, let $\PP_\eta$ denote the measure for critical site percolation on the
 triangular lattice. Let $\phi>0$, $f:\Ball_{1+\phi}\ra \hat{\CC}$ be a conformal map, and $D:=f(\Ball_1)$.
    
  Then the laws of $\mathscr{B}^0_{D}$ and $f(\mathscr{B}^0_{\Lambda_1})$ are identical, where
  \begin{align*}
    f(\mathscr{B}^0_{\Lambda_1}) :=\{f(\calB)\,:\,\calB\in\mathscr{B}^0_{\Lambda_1}\}.
  \end{align*}
\end{thm}

\medskip
In addition to the convergence of arm measures, 
\cite{Garban2010} contains a proof of the conformal covariance of these measures.
The relevant result is Theorem 6.7 in \cite{Garban2010}, stated below.
\begin{thm} \label{assu:covar} 
 For $\eta \ge 0$, let $\PP_\eta$ denote the measure for critical site percolation on the triangular lattice.
 Let $D\subseteq \CC$ be a domain and $f: D \to \CC$ be a conformal map. Let $A\subset \CC$ be a proper
 annulus with piecewise smooth boundary with $\ol A\subset D$. 
 For a Borel set $B\subseteq f(D)$, let
 \begin{align*}
  \mu_{1,A}^{0*}(B) := \int_{f^{-1}(B)} |f'(z)|^{2-\alpha_1}d\mu_{1,A}^{0}(z).
 \end{align*}
 Then the laws of
 $\mu^0_{1,f(A)} = \mu^0_{1,f(A)}(\omega_{0,f(D)})$ and $\mu^{0*}_{1,A} = \mu^{0*}_{1,A}(\omega_{0,D})$
 coincide.
\end{thm}

The boundedness of $f'$ discussed earlier implies that approximating the cluster measures
in $D$ by one-arm measures of annuli of the form $f(\Lambda_{\delta/2}\setminus \Lambda_{\varepsilon/2})$ provides
the same limit as approximating the same measures by one-arm measures of annuli of the form
$\Lambda_{\delta/2}\setminus \Lambda_{\varepsilon/2}$. Hence, one can carry out the arguments in the proof of
Lemma \ref{lem:bound_err} using one-arm measures of annuli of the form $f(\Lambda_{\delta/2}\setminus \Lambda_{\varepsilon/2})$.
This observation and Theorem \ref{assu:covar} imply the following result, where $\tilde\muColl^0_{D}$ denotes
the collection of measures of all clusters in $\mathscr{B}^0_{D}$, and $\mu^{0*}$ is defined in Theorem \ref{thm:cinvar}.
\begin{thm} \label{thm:covar_prec}
  For $\eta \ge 0$, let $\PP_\eta$ denote the measure for critical site percolation on the
  triangular lattice. Let $\phi>0$, $f:\Ball_{1+\phi}\ra \hat{\CC}$ be a
  conformal map, and $D:=f(\Ball_1)$. Then
  the laws of $\tilde\muColl^0_{D}$ and $f(\tilde\muColl^0_{\Lambda_1})$ 
  are identical, where 
  \begin{align*}
    f(\tilde\muColl^0_{\Lambda_1}) :=\{\mu^{0*}\,:\,\mu^0\in \tilde\muColl^0_{\Lambda_1}\}.
  \end{align*}
\end{thm}

We conclude this section with a brief discussion of the proof of Theorem~\ref{thm:invar}.

\begin{proof}[Proof of Theorem \ref{thm:invar}]
The theorem follows from a straightforward modification of the arguments above, using the rotation
and translation invariance and scaling covariance of the $1$-arm measures under
Assumptions \ref{assu:markov} - \ref{assu:existence_full_confInvLimit}, which follow easily from the proof
of Theorem \ref{assu:conv_measures} (see also \cite[Equation (6.1) and Proposition 6.4]{Garban2010}).
\end{proof}

\section{Proof of the convergence of the largest Bernoulli percolation clusters}\label{sec:pf_of_conv_largest_clusters}
Now we turn to the precise version and to the proof of Theorem \ref{thm:largest_Clusters}.
\begin{thm}\label{thm:conv_Largest_clusters_precise}
 Let $\PP$ be a coupling such that $(\omega_\eta,L_{\eta}) \rightarrow (\omega_0,L_0)$ a.s. as $\eta \rightarrow 0$.
 Then for all $i \in \NN$ the $i$-th largest cluster $\calM_{(i)}^{\eta}$ converges in $\PP$-probability to
 $\calM_{(i)}^{0}$ as $\eta \to 0$, where $\calM_{(i)}^{0}$ is a measurable function of $(\omega_{0}, L_0)$.
 In particular, $(\omega_{\eta},L_\eta,\calM_{(i)}^{\eta}) \to (\omega_{0},L_0,\calM_{(i)}^{0})$ in distribution.
 The same convergence holds for the measures $\mu_{\calM_{(i)}^{\eta}}^{\eta}$.
\end{thm}

Let us start with some preliminary results. Recall the definition of
collections of (portions of) clusters $\mathscr{B}_{\Ball_1}^{\eta}(\delta)$
in Section \ref{sec:clustersinSubset}.
\begin{prop}\cite[Proposition 3.2]{BC13}\label{prop:BC_gaps_cl}
 Let $\delta \in (0,1)$. For all $\varphi > 0$ there exist $\eta_{0}, \alpha > 0$ such that, for all $\eta < \eta_{0}$,
 \[
  \PP_{\eta}(\exists \calB, \calB' \in \mathscr{B}_{\Ball_1}^{\eta}(\delta): \calB \neq \calB': |\mu_{\calB}^{\eta}(\Ball_1) - \mu_{\calB'}^{\eta}(\Ball_1)| < \alpha) < \varphi.
 \]
\end{prop}

\begin{proof}[Proof of Proposition \ref{prop:BC_gaps_cl}]
In \cite{BC13} a proof for Proposition \ref{prop:BC_gaps_cl} was given for bond percolation on the square lattice,
however the proof also works for other models, like site percolation on the triangular lattice, as noted in Remark (i)
after Theorem 1.1 in \cite{BC13}.
\end{proof}

\begin{lemma}[Lemma 4.4 of \cite{BCKS01}] \label{lem:BCKS2}
 There are positive constants $c,C$ such that for all $x,y>0$
   \[
    \PP_{\eta}(\exists \calB \in \mathscr{B}_{\Ball_1}^{\eta}: \mu_{\calB}^{\eta}(\Ball_1) > x \text{ and } \diam(\calB)<y) < Cy^{-1}\exp(-cx/\sqrt{y})
   \]
   for all $\eta < \eta_0 = \eta_0(x,y)$.
\end{lemma}

The next proposition follows easily from a combination of Lemma \ref{lem:BCKS2} and \cite[Theorems 3.1, 3.3 and 3.6]{BCKS01}
(see also \cite{BC13}).
\begin{prop}\label{prop:Large_Volume_large_diam}
 Let $i\in \NN$ be fixed. For all $\varphi > 0$ there exist $\delta >0, \eta_{0} > 0$ such that, for all $\eta < \eta_{0}$,
 \[
  \PP_{\eta}(\exists j \le i: \calM_{(j)}^{\eta} \not\in \mathscr{B}_{\Ball_1}^{\eta}(\delta)) < \varphi.
 \]
\end{prop}

\begin{proof}[Proof of Theorem \ref{thm:conv_Largest_clusters_precise}]
 Let $i\in \NN$ be fixed and $\PP$ be a coupling such
 that $(\omega_{\eta},L_\eta) \to (\omega_{0},L_0)$ a.s. as $\eta\ra0$. First we
 show that the $i$-th largest clusters in the scaling limit can
 almost surely be defined as a function of the pair $(\omega_{0},L_{0})$. Then we show that the $i$-th largest cluster
 $\calM_{(i)}^{\eta}$ in the discrete configuration $\omega_{\eta}$
 converges to the $i$-th largest continuum cluster.
 
 Let $m\in \NN$. Theorems \ref{thm:strong_Hausdorff_in_domain} 
 and \ref{thm:strong_meas} show that the sequence of
 clusters $\underline\calB_{\Ball_1}^{0}(3^{-m})$ and their corresponding measures
 $\underline\mu^{0}(3^{-m})$ are a.s. well defined. 
 
 We define the \emph{volume} of a continuum cluster $\calB\in\mathscr{B}_{\Ball_1}^0$ as
 $\mu_{\calB}^{0}(\Ball_{1})$.
 Lemma \ref{lem:BCKS} shows that the volumes of the clusters
 $\calB\in\mathscr{B}_{\Ball_1}^0(3^{-m})$ are a.s. finite. Moreover, Lemma \ref{lem:n_0},
 together with the tightness of the number of excursions in $\Ball_{1}$ of diameter
 at least $3^{-m}$, gives that
 $h^{0}(3^{-m}) := |\mathscr{B}_{\Ball_1}^0(3^{-m})|$ is a.s. finite. 
 Thus we can reorder the sequence  of clusters $\underline\calB^{0}_{\Ball_{1}}(3^{-m})$ in decreasing order by their volume.
 We break ties in some deterministic way. However, we will see below that ties occur with probability $0$. Let $\calM_{(j)}^0(3^{-m})$ denote the $j$-th cluster in this new
 ordering.
  
 Let $\varphi>0$ be arbitrary and take $\alpha$ and $\eta_0$ as in Proposition \ref{prop:BC_gaps_cl}.
 Then, for $\eta < \eta_0$,
 \begin{multline}\label{eq:pf:largeCl}
  \PP(\exists \calB,\calB' \in \mathscr{B}_{\Ball_1}^0(3^{-m}): \calB\neq\calB', |\mu_{\calB}^{0}(\Ball_1) - \mu_{\calB'}^{0}(\Ball_1)| < \alpha/2)\\
  \begin{aligned}
    \le\, &\PP( \exists \calB, \calB' \in\mathscr{B}_{\Ball_1}^{\eta}(3^{-m}): \calB \neq \calB', |\mu_{\calB}^{\eta}(\Ball_1) - \mu_{\calB'}^{\eta}(\Ball_1)| < \alpha )\\
    &+ \PP(\exists j \le h^{0}(3^{-m}): |\mu_{\calB^\eta_{j}}^{\eta}(\Ball_1) - \mu_{\calB^0_{j}}^{0}(\Ball_1)| > \alpha/4)\\
    \le\, &\varphi+ \PP(\exists j \le h^{0}(3^{-m}): |\mu_{\calB^\eta_{j}}^{\eta}(\Ball_1) - \mu_{\calB^0_{j}}^{0}(\Ball_1)| > \alpha/4).  
  \end{aligned}
 \end{multline}
 The second term in the last line of \eqref{eq:pf:largeCl} tends to $0$ as $\eta\ra0$, since $h^{0}(3^{-m})$ is a.s. finite
 and $\underline\mu^{\eta}(3^{-m})\ra \underline\mu^{0}(3^{-m})$
 in probability by Theorem \ref{thm:strong_meas}. Since $\varphi>0$ was arbitrary, this shows that 
 \[
   \PP(\exists \calB,\calB' \in \mathscr{B}^0_{\Ball_1}(3^{-m}): \calB\neq\calB', |\mu_{\calB}^{0}(\Ball_1) - \mu_{\calB'}^{0}(\Ball_1)| =0)=0.
 \]
 That is, with probability $1$, there are no ties in the ordering of continuum clusters described above .
 
 Now we show that, for all $j \le i$,
 \begin{align} \label{eq:pf_LC_1}
   \PP(\exists m_0\in \NN \text{ s.t. } \calM_{(j)}^{0}(3^{-m_0}) = \calM_{(j)}^{0}(3^{-m}) \text{ for all } m\geq m_0) = 1. 
 \end{align}
Consider the event
 \begin{align*}
   E = \{\exists j_0 \leq i: \nexists m_0\in \NN \text{ s.t. } \calM_{(j_0)}^{0}(3^{-m_0}) = \calM_{(j_0)}^{0}(3^{-m}) \text{ for all } m\geq m_0\}
 \end{align*}
and the events
\begin{align*}
E^m_n = \{ \exists \calB \in \mathscr{B}^0_{\Ball_1}(3^{-m}) \text{ s.t. } \diam(\calB) < 3^{-m+1} \text{ and }
\mu^{0}_{\calB}(\Lambda_1) >1/n\}.
\end{align*}
Note that
\begin{align*}
E \subset \bigcup_{n=1}^{\infty} \{ E^m_n \text{ for infinitely many } m \in \NN \}.
\end{align*} 
Theorems \ref{thm:strong_Hausdorff} and \ref{thm:strong_meas} and Lemma \ref{lem:BCKS2} imply that, for each $m,n \geq 1$,
there is $\eta_0 = \eta_0(m,n)$ such that
\begin{eqnarray*}
\PP(E^m_n) & \leq & 2 \PP(\exists \calB \in \mathscr{B}^{\eta}_{\Lambda_1}: \mu^{\eta}_{\calB}(\Lambda_1) > 1/(2n) \text{ and } \diam(\calB)<6 \times 3^{-m}) \\
& \leq & C 3^{m-1} \exp\Big(-\frac{c}{2\sqrt{6}n}3^{m/2} \Big)
\end{eqnarray*}
for all $\eta \leq \eta_0$.
Since $\sum_{m=1}^{\infty} 3^{m} \exp\Big(-\frac{c}{2\sqrt{6}n}3^{m/2} \Big) < \infty$, it follows from the Borel-Cantelli Lemma that
$\PP(E^m_n \text{for infinitely many } m \in \NN) = 0$ for every $n \geq 1$. Hence $\PP(E)=0$, which proves \eqref{eq:pf_LC_1}.

For each $j\leq i$, we set $\calM_{(j)}^{0}:=\calM_{(j)}^{0}(3^{-m_0})$, where $m_0$ is as in the event on the left hand side
of \eqref{eq:pf_LC_1}. It remains to show that $\calM_{(i)}^{\eta}$ converges in probability to $\calM_{(i)}^{0}$, as well as the
analogous statement for their measures. Let $\varepsilon,\alpha > 0$ and $m>0$. First we check that 
 \begin{multline} \label{eq:pf_LC_2}
  \PP(d_{H}(\calM_{(i)}^{\eta},\calM_{(i)}^{0}) > \varepsilon)\\
  \begin{aligned}
    \leq & \,\, \PP(\exists j\leq i: \calM_{(j)}^{0}\neq \calM_{(j)}^{0}(3^{-m}))\\
    &+\PP(\exists j\leq i: \calM_{(j)}^{\eta}\neq \calM_{(j)}^{\eta}(3^{-m}))\\
    & + \PP(\exists \calB,\calB' \in \mathscr{B}^0_{\Ball_1}(3^{-m}): \calB\neq\calB', |\mu_{\calB}^{0}(\Ball_1) - \mu_{\calB'}^{0}(\Ball_1)| < \alpha)\\
    & +\PP(\exists \calB,\calB' \in \mathscr{B}^\eta_{\Ball_1}(3^{-m}): \calB\neq\calB', |\mu_{\calB}^{\eta}(\Ball_1) - \mu_{\calB'}^{\eta}(\Ball_1)| < \alpha)\\
    & + \PP(\exists k \le h^{0}(3^{-m}): |\mu_{\calB_{k}^{\eta}}^\eta(\Ball_1) - \mu_{\calB_{k}^{0}}^0(\Ball_1)|> \alpha/3)\\    
    &+ \PP(\exists k \le h^{0}(3^{-m}): d_{H}(\calB_{k}^{\eta},\calB_{k}^{0}) > \varepsilon),
  \end{aligned}
 \end{multline}
 where $\calB^\eta_{k}$ (resp. $\calB^0_{k}$) is the $k$-th cluster of $\calB^\eta_{\Ball_1}(3^{-m})$
 (resp. $\mathscr{B}^0_{\Ball_1}(3^{-m})$)
in the order used in the proofs of Theorems \ref{thm:strong_Hausdorff} and \ref{thm:strong_meas}.
 
 We justify \eqref{eq:pf_LC_2} as follows. On the complement of the first two events on the right
 hand side of \eqref{eq:pf_LC_2}, the $i$-th largest clusters at scale $\eta$ and $0$ (i.e., in the scaling limit)
 have diameter at least $3^{-m}$. On the complement of the third and fourth event on the right hand side of \eqref{eq:pf_LC_2},
 the normalized volumes of the different clusters with diameter at least $3^{-m}$ are at least $\alpha$ apart
 at both scales $\eta$ and $0$. Thus, on the complement of the first five events on the right hand side of
 \eqref{eq:pf_LC_2}, the ordering according to their volume of the $k$ largest clusters at scale $\eta$ and $0$ coincide;
 that is, for all $j\leq i$, there is a unique $k_j\leq h^0(3^{-m})$ such that $\calM_{(j)}^\eta = \calB^\eta_{k_j}$
 and $\calM_{(j)}^0 = \calB^0_{k_j}$. This, together with the last term in the right hand side of \eqref{eq:pf_LC_2},
 proves \eqref{eq:pf_LC_2}.
 
 Let $\varphi >0$ be arbitrary. By \eqref{eq:pf_LC_1} and Proposition \ref{prop:Large_Volume_large_diam}, we find $m$ and
 $\eta_0>0$ such that the first and second term on the right hand side of \eqref{eq:pf_LC_2} are less than $\varphi/6$ for
 all $\eta <\eta_0$. Then we use the bounds in \eqref{eq:pf:largeCl} and Proposition \ref{prop:BC_gaps_cl} and find $\alpha, \eta_1>0$
 so that the third and fourth term on the right hand side of \eqref{eq:pf_LC_2} are less than $\varphi/6$ for all $\eta <\eta_1$. 
 Finally, we apply Theorem \ref{thm:strong_meas} to control the fifth term and Theorem \ref{thm:strong_Hausdorff} to control
 the sixth term, and deduce that $\limsup_{\eta\ra0} \PP(d_{H}(\calM_{(i)}^{\eta},\calM_{(i)}^{0}) > \varepsilon) < \varphi$.
 Since $\varphi$ and $\ve$ were arbitrary, this shows that $\calM_{(i)}^{\eta}\ra\calM_{(i)}^{0}$ in probability as $\eta\ra0$.
 
 The proof for the convergence of normalized counting measures goes in a similar way: notice that if we replace the
 fifth term on the right hand side of \eqref{eq:pf_LC_2} with
 \begin{align*}
   \PP(\exists j \le h^{0}(3^{-m}): d_p(\mu_{\calB_{j}^{\eta}}^\eta, \mu_{\calB_{j}^{0}}^0)> \alpha/3),
 \end{align*}  
 then we get an upper bound for the probability $\PP(\exists j\leq i: d_P(\mu^\eta_{\calM^\eta_{(j)}},\mu^0_{\calM^0_{(j)}})>\alpha/3)$.
 This completes the proof of Theorem \ref{thm:conv_Largest_clusters_precise}.
\end{proof}

\newcommand{\etalchar}[1]{$^{#1}$}
\providecommand{\bysame}{\leavevmode\hbox to3em{\hrulefill}\thinspace}
\providecommand{\MR}{\relax\ifhmode\unskip\space\fi MR }
\providecommand{\MRhref}[2]{%
  \href{http://www.ams.org/mathscinet-getitem?mr=#1}{#2}
}
\providecommand{\href}[2]{#2}

\end{document}